\documentclass[12pt]{amsart}
\usepackage{amsmath}
\usepackage{amsfonts,amssymb}
\usepackage{euscript}
\usepackage{textcomp}
\usepackage{amsthm}
\usepackage[matrix,arrow,curve]{xy}
\usepackage{array}

\numberwithin{equation}{section}

\theoremstyle{plain}
\newtheorem{theorem}{Theorem}[section]
\newtheorem{lemma}[theorem]{Lemma}
\newtheorem{predl}[theorem]{Proposition}
\newtheorem{corollary}[theorem]{Corollary}

\newtheorem{conjecture}[theorem]{Conjecture}

\theoremstyle{definition}
\newtheorem{definition}[theorem]{Definition}
\newtheorem{remark}[theorem]{Remark}
\newtheorem{example}[theorem]{Example}

\newcommand{\R}{\mathbb R}
\newcommand{\N}{\mathbb N}
\newcommand{\Z}{\mathbb Z}

\renewcommand{\AA}{\mathcal A}
\newcommand{\BB}{\mathcal B}

\renewcommand{\O}{\mathcal O}
\renewcommand{\k}{\mathsf k}

\newcommand{\xra}{\xrightarrow}
\renewcommand{\le}{\leqslant}
\renewcommand{\ge}{\geqslant}
\newcommand{\bul}{\bullet}

\DeclareMathOperator{\Hom}{\textup{Hom}}

\DeclareMathOperator{\Ext}{\textup{Ext}}

\DeclareMathOperator{\End}{\mathrm{End}}

\DeclareMathOperator{\Spec}{\mathrm{Spec}}

\DeclareMathOperator{\coh}{\mathrm{coh}}

\DeclareMathOperator{\modd}{\mathrm{mod{-}}}

\DeclareMathOperator{\Perf}{\mathrm{Perf}}

\DeclareMathOperator{\LSdim}{\underline{\mathrm{Sdim}}}
\DeclareMathOperator{\USdim}{\overline{\mathrm{Sdim}}}
\DeclareMathOperator{\Ldim}{-\,\underline{\mathrm{dim}}}
\DeclareMathOperator{\Udim}{-\,\overline{\mathrm{dim}}}
\DeclareMathOperator{\Rdim}{Rdim}
\DeclareMathOperator{\gldim}{\mathrm{gldim}}
\DeclareMathOperator{\Ddim}{Ddim}

\DeclareMathOperator{\op}{\mathrm{op}}

\newcommand{\bbL}{{\mathbb L}}
\newcommand{\bbR}{{\mathbb R}}

\newcommand{\bbZ}{{\mathbb Z}}

\newcommand{\cO}{{\mathcal O}}

\newcommand{\cA}{{\mathcal A}}

\newcommand{\cE}{{\mathcal E}}

\newcommand{\cS}{{\mathcal S}}
\newcommand{\cT}{{\mathcal T}}

\author{Alexey Elagin}
\address{Institute for Information Transmission Problems (Kharkevich Institute), Moscow, RUSSIA;
National Research University Higher School of Economics, 
Russian Federation
}
\email{alexelagin@rambler.ru}

\author{Valery A.~Lunts}

\address{
  Department of Mathematics,
  Indiana University,
  Rawles Hall,
  831 East 3rd Street,
  Bloomington, IN 47405,
  USA;
  National Research University Higher School of Economics, Russian Federation
}

\email{vlunts@indiana.edu}

\title[Three notions of dimension for triangulated categories]
{Three notions of dimension for triangulated categories}

\begin{document}

\begin{abstract}
In this note we discuss three notions of dimension for triangulated categories: Rouquier dimension,  diagonal dimension and Serre dimension.  We prove some basic properties of these dimensions, compare them and discuss open problems.
\end{abstract}

\maketitle

\footnotetext[0]{The study of the first author has been funded within the framework of the HSE University Basic Research Program and the Russian Academic Excellence Project '5-100'.
The second author is partially supported by Laboratory of Mirror Symmetry NRU HSE, RF Government grant, ag. \textnumero 14.641.31.0001.}

\section{Introduction}

We discuss three notions of dimension for triangulated categories: Rouquier dimension,   diagonal dimension and Serre dimension. The Rouquier dimension $\Rdim$ was defined in \cite{Rou}, the diagonal dimension $\Ddim$ appeared (in the context of algebraic varieties) in \cite{BaFa}, and the Serre dimension is a special case of the more general notion of dimension of a category with respect to an endofunctor, which may be extracted from  \cite{DHKK}. We distinguish between lower Serre dimension, denoted $\LSdim$, and upper Serre dimension, denoted $\USdim$. Our main focus is on triangulated categories $\Perf (A)$ for a smooth and compact dg algebra $A$ over a field $\k$. These are categories for which all three dimensions make sense.

The purpose of this article is to study some basic properties of these dimensions, to compare them, and to formulate some open problems. We propose the following four test problems/questions:

\medskip

\begin{enumerate}
\item Classify categories of dimension zero.

\item Monotonicity: if $\Perf(A)$ is a full subcategory of $\Perf(B)$, is $\dim \Perf(A)\le \dim \Perf(B)$?

\item Additivity: is $\dim \Perf(A\otimes B)=\dim \Perf(A)+ \dim\Perf(B)$?

\item  Behavior in families: given a family of categories $\Perf(A_x),\ x\in \Spec R$, how does the function
$\dim \Perf(A_x)$ on $\Spec R$ behave? Is it semicontinuous?
\end{enumerate}

\medskip

1) We can completely classify categories with diagonal dimension zero (Proposition \ref{prop-diag-dim-zero}) or Rouquier dimension zero (by results of Hanihara~\cite{Han} and Amiot~\cite{Am}).
For Serre dimensions we only have partial results and conjectures (Conjecture \ref{conj-serre-dim-zero}, Example~\ref{example_xyz}).

2) Monotonicity holds for the Rouquier and diagonal dimension (Subsection \ref{subs-monot}, Proposition \ref{monot-for-diag-dim}). Monotonicity does not hold for Serre dimension, see Subsection~\ref{section_serremonot}. It seems to be an interesting problem to modify the definition so that Serre dimension would be monotonous (under certain conditions).

3) Additivity holds for the Serre dimension (Proposition \ref{addit-for-serre-dim}) and does not in general hold for the other two (Example \ref{example_maintable}). For the diagonal dimension we prove the inequality $\Ddim  \Perf(A\otimes B)\le \Ddim  \Perf(A)+\Ddim  \Perf(B)$ (Proposition \ref{add-for-diag-dim}, Example \ref{ex-strict}) and there exist examples with the opposite strict inequality for the Rouquier dimension: $\Rdim  \Perf(A\otimes B)> \Rdim  \Perf(A)+\Rdim  \Perf(B)$ (Example \ref{example_maintable}).

4) Semicontinuity of the Serre dimension holds when we have a family of algebras, not just dg algebras (Theorem \ref{theorem_serre-in-families}). We have no results about the other two dimensions. But we expect the upper semicontinuity to hold in general for Rouquier and diagonal dimension.

For a dg algebra $A$, the Rouquier and diagonal dimensions compare as follows (see Proposition \ref{ddim-geq-rdim})
$$\Rdim  \Perf(A)\le \Ddim  \Perf(A).$$
Also we make the following 

\begin{conjecture}[=Conjecture \ref{main-conj}] 
\label{main-cinj-intr}
Let $A$ be a smooth and compact dg algebra over a field. Suppose 
$\LSdim \Perf(A)=\USdim \Perf(A)$.
Then
\begin{equation}
\label{eq_4dim}
\Rdim  \Perf(A)\le \USdim  \Perf(A)\le \Ddim  \Perf(A).
\end{equation}
\end{conjecture}
The first inequality in \eqref{eq_4dim} would in particular confirm the expectation that for a smooth projective variety $X$ the Rouquier dimension of the category $D^b(\coh X)$ is equal to the dimension of~$X$, see \cite{Or}.


We should note that the Rouquier and diagonal dimensions are by definition nonnegative integers (or $+\infty$). Apriori both Serre dimensions are just  real numbers. They do not have to be an integer (Example \ref{example_maintable}) and can be negative (see Example \ref{neg-serre} for the category with $\LSdim=\USdim<0$ and Example \ref{example_orlov} of a smooth and compact dg algebra $A$ with $\LSdim\Perf(A)<0$). We do not know of an example with an irrational lower or upper Serre dimension. For nice categories we expect a reasonable behavior of $\LSdim $ and $\USdim$:

\begin{conjecture}[=Conjecture \ref{conj-pos-rat}]  
Let $A$ be a smooth and compact dg algebra over a field $\k$. Then $\LSdim  \Perf(A)$ and $\USdim  \Perf(A)$ are rational numbers, and $\USdim \Perf(A)$ is nonnegative.
\end{conjecture}

This paper has a companion \cite{El}, where the three dimensions are computed in many interesting examples. Therefore we decided to include a minimal number of examples in the present text.

The structure of the text is the following. Section \ref{section_background} introduces notation and contains some well known background material on triangulated categories, enhancements, Krull-Schmidt categories, etc. In Section \ref{section_rouquier} we discuss the Rouquier dimension. Section \ref{section_diagonal} introduces the diagonal dimension. In Section \ref{section_serre} we study the Serre dimension. Section \ref{section_general-dim} is devoted to the more general notion of a dimension of a category with respect to an endo-functor which is suggested by \cite{DHKK}. Finally in Section \ref{section_compare} we try to compare the three dimensions and give some examples. Conjectures and open questions are scattered throughout the text.

We are grateful to Dmitri Orlov for very useful discussions. Examples \ref{example_orlov} and \ref{example_F3} are due to him. Also we thank an anonymous referee for the helpful remarks on the manuscript and for the suggestion that Th. 5.17(b) holds for any negatively graded dg algebra.

\section{Background material and notation}
\label{section_background}
Here we recall some simple and well known facts about triangulated categories, semi-orthogonal decompositions, enhancements, Krull-Schmidt categories etc. The general references on the subject include \cite{BK1},\cite{BK2},\cite{BLL},\cite{Dr},\cite{ELO},\cite{Ke},\cite{To}.

\subsection{Some notation}
Usually we work over a field $\k$. All algebras, dg algebras, schemes are $\k$-algebras, dg $\k$-algebras, $\k$-schemes, etc.\,unless stated otherwise. Symbol $\otimes$ means $\otimes_{\k}$. Similarly, $\dim$ means $\dim_{\k}$ (dimension of a $\k$-vector space).

For objects $X,Y$ of a triangulated category we use the standard notations
$$\Hom^i(X,Y)=\Hom(X,Y[i])\qquad\text{and}\qquad \Hom^{\bul}(X,Y):=\oplus_i\Hom^i(X,Y),$$
the latter is treated as a graded abelian group. For  dg modules $X,Y$ over a dg algebra~$A$ we use  notation $\Hom_A(X,Y)$  for the dg $\k$-module of homomorphisms: its $i$-th component is the space  of homomorphisms $X\to Y[i]$ as graded $A$-modules.

For a Noetherian algebra $A$ we denote by $\mathrm{mod-}A$ the abelian category of finitely generated right $A$-modules. Note that if $A$ has finite global dimension, then $D^b(\mathrm{mod-}A)\simeq \Perf(A)$.

\subsection{Generation of triangulated categories}
Let $T$ be a triangulated category and let $A$ and $B$ be two subsets of objects in $T$. We denote by $A\star B$ the subset of objects $F$ of $T$ such that there exists a distinguished triangle
$$X\to F \to Y\to X[1]$$
with $X\in A$, $Y\in B$. This operation $\star$ was introduced in \cite[p.33]{BBD}, where it was shown to be associative.

We will be interested in the case when a subset $A$ consists of all objects isomorphic to a finite direct sum of shifts of a given object $E$ (so in particular $A$ contains the zero object). We then denote $A=[ E] _0$.
Define $[E] _n$ as the $n$-fold star product
$$[E]_n:=[E]_0\star \ldots \star [E]_0 \quad (n+1\quad \text{factors}).$$
Finally let $\langle E\rangle _n$ be the full subcategory of $T$ consisting of objects which are direct summands of objects in   $[E]_n$. 
We also put $\langle E\rangle :=\cup _n\langle E\rangle _n\subset T$. It is a full thick triangulated subcategory of $T$. (Note that our indexing of the categories $\langle E\rangle _n$ differs by one from that used in \cite{BVdB} and \cite{Rou}. We count the number of cones used.)

\begin{definition} 
An object $E$ is called a classical \emph{generator} of $T$ if $T=\langle E\rangle $. It is
called a \emph{strong generator} if $T=\langle E\rangle _n$ for some $n$.
Saying about generators or generation in this text we always mean classical generators.
\end{definition}

The notion of a strong generator was introduced in \cite{BVdB}, where it was shown that the category $D^b(\coh X)$ has a strong generator if $X$ is a smooth and proper $\k$-scheme.

\begin{remark} \label{rem-gen-loc} Let $T$ be a triangulated category and let $l\colon T\to T^\prime $ be a Verdier localization. If $E\in T$ is a generator, then $l(E)\in T^\prime$ is also a generator. Moreover if $T=\langle E\rangle _n$ for some object $E\in T$, then $T^\prime =\langle l(E)\rangle _n$.
\end{remark}

\subsection{Enhancements of triangulated categories}

It is convenient (and for some purposes necessary) to work with enhancements of triangulated categories. We recall the definition in \cite{BLL}. As usual for a dg category $\cA$ we denote by $[\cA]$ its homotopy category \cite{To}.

\begin{definition} An \emph{enhancement} of a triangulated category $T$ is a pair $(\cT ,\phi)$, where $\cT$ is a pre-triangulated dg category and $\phi \colon [\cT]\to T$ is an equivalence of triangulated categories.
\end{definition}

\begin{remark} \label{induced-enhancement} Note that if $T$ is a triangulated category and $S\subset T$ is a full triangulated subcategory, then an enhancement of $T$ induces an enhancement of $S$. Indeed, if $\cT$ is a pre-triangulated category that enhances $T$ take its full dg subcategory $\cS$ consisting of objects which are mapped to $S$ by $\phi$. Then $\cS$ is an enhancement of $S$.
\end{remark}

Let $A$ be a dg algebra. Denote by $Mod-A$ the dg category of right dg $A$-modules. It is a pre-triangulated dg category. The localization of the triangulated category $[Mod-A]$
by the subcategory of acyclic dg modules is the derived category $D(A)$ of $A$. The full dg subcategory $hproj-A\subset Mod-A$ consisting of h-projective dg modules is an enhancement of the category $D(A)$, i.e. the natural triangulated functor $[hproj-A]\to D(A)$ is an equivalence.

Recall that triangulated category $\Perf(A)$ is defined as the full thick triangulated subcategory of $D(A)$ generated by the dg $A$-module $A$. It can be intrinsically defined as the category of compact objects in $D(A)$ \cite{Ke}. Denote by $P erf (A)$ the full dg category of $hproj-A$ consisting of objects that belong to $\Perf(A)$. Then $P erf (A)$ is an enhancement of $\Perf(A)$. The functor 
$$(-)^\vee :=\bbR \Hom _A(-,A)$$ 
defines
an anti-equivalence of $\Perf(A)$ and $\Perf(A^{\op})$ such that $(-)^{\vee \vee}$ is isomorphic to the identity functor.

\begin{lemma} \label{lemma-left-right} Let $M,N\in \Perf(A)$. Then there is a quasi-isomorphism of dg $\bbZ$-modules
$$M\stackrel{\bbL}{\otimes }_AN^\vee \simeq \bbR \Hom _A(N,M).$$
Hence an isomorphism of graded abelian groups
$$H^\bullet (M\stackrel{\bbL}{\otimes }_AN^\vee)\simeq \Hom^\bullet_{\Perf(A)} (N,M).$$
\end{lemma}

\begin{proof} We may and will assume that $M$ and $N$ are in $Perf (A)$, which allows us to write $\otimes _A$ and $\Hom _A$ instead of $\stackrel{\bbL}{\otimes _A}$ and $\bbR \Hom _A$ respectively. Notice that we have a natural morphism of $\bbZ$-complexes
$$M\otimes _A\Hom _A(N,A)\to \Hom _A(N,M),\quad m\otimes f\to (n\mapsto mf(n)).$$
This morphism is an isomorphism of complexes if $M=A$, hence it is a quasi-isomorphism for all $M,N\in Perf (A)$.
\end{proof}

\begin{lemma} \label{ess-lemma} Let $\cT$ be a pre-triangulated dg category and let $E$ be an object of $\cT$. Consider the full thick triangulated subcategory $\langle E\rangle \subset [\cT]$ which is generated by $E$. Denote by $\langle E\rangle _{\cT}\subset \cT$ the full pre-triangulated dg subcategory whose objects belong to $\langle E\rangle$ (so that $\langle E\rangle _{\cT}$ is an enhancement of $\langle E\rangle$). Let $\cE :=\End _{\cT}(E)$ be the endomorphism dg algebra of $E$. Assume that the category $\langle E\rangle$ is Karoubian. Then the natural dg functor
$$\psi _E\colon \cT \to Mod-\cE,\quad X\mapsto \Hom _{\cT}(E,X)$$
maps $\langle E\rangle _{\cT}$ to the dg category $Perf (\cE)$ and is a quasi-equivalence of these categories. Hence in particular it induces an equivalence of tringulated categories $[\psi _E]\colon \langle E\rangle \to \Perf(\cE)$.
\end{lemma}

\begin{proof} The dg functor $\psi _E$ maps $E$ to $\cE$. Hence it maps $\langle E\rangle _{\cT}$ to $Perf (\cE)$. By definition it induces an isomorphism
$$\Hom^{\bul}_{\langle E\rangle}(E,E)\simeq \Hom ^{\bul} _{\Perf(\cE)}(\cE,\cE).$$
Hence $[\psi _E]$ is full and faithful. Since both categories $\langle E\rangle$ and $\Perf(\cE)$ are Karoubian, it follows that $[\psi _E]$ is also essentially surjective. This implies that $\psi _E\colon \langle E\rangle _{\cT}\to Perf (\cE)$ is a quasi-equivalence.
\end{proof}

We obtain an immediate corollary.

\begin{corollary} \label{cor-quasi-equiv} Let $T$ be a triangulated category with an enhancement $\cT$. Assume that $T$ is Karoubian and has a generator. Then the dg category $\cT$ is quasi-equivalent to the dg category $Perf (\cE)$ for a dg algebra $\cE$. Hence in particular there is an equivalence of triangulated categories
$T\simeq \Perf(\cE)$.
\end{corollary}

\begin{proof} Indeed, we have  $\langle E\rangle =T$ in the notation of Lemma \ref{ess-lemma}.
\end{proof}

\subsection{Triangulated categories of dg modules over dg algebras}
Let $R$ be a commutative ring. Here we collect some fact about dg $R$-algebras, modules over such algebras and their derived categories. Such a general setting will be needed in Subsection~\ref{section_serreinfamilies} and in Theorem~\ref{theorem_generaliz-serre-in-families}, for the remaining part of the paper one can assume that $R$ is a field.

Recall that an $R$-complex $F\in D(R)$ is \emph{perfect} if it is locally on $\Spec R$ quasi-isomorphic to a finite complex of free $\cO _{\Spec R}$-modules of finite rank. It is known that a perfect $R$-complex is quasi-isomoprhic to a \emph{strict perfect} $R$-complex, i.e. to a finite complex of finitely generated projective $R$-modules \cite[Theorem 2.4.3]{ThTr}. It follows that a perfect $R$-complex which is h-projective, is homotopy equivalent to a strict perfect $R$-complex.


\emph{In this section all dg $R$-algebras will be assumed to be $R$-h-projective.}


The following definition is due to Kontsevich.

\begin{definition} Let $A$ be a ($R$-h-projective) dg $R$-algebra. Then
\begin{enumerate}
\item  $A$ (\emph{homologically}) \emph{smooth} if $A\in \Perf(A^{\op}\otimes _RA)$,

\item $A$ is \emph{compact} if it is perfect as an $R$-complex.
\end{enumerate}
\end{definition}

\begin{example} Suppose $R=\k$ is a field. Then any dg $R$-algebra $A$ is $R$-h-projective, and $A$ is compact if and only if $\dim\oplus_i H^i(A)<\infty$.
\end{example}

\begin{example} Let $A$ be an $R$-algebra which is a finitely generated projective $R$-module. Then $A$ considered as a dg $R$-algebra is $R$-h-projective and compact.
\end{example}

\begin{definition} A \emph{family} of smooth and compact triangulated categories over $\Spec R$ is by definition the category $\Perf(A)$ for a smooth and compact dg $R$-algebra $A$.
\end{definition}

\begin{lemma}
\label{perf-r-compl} 
\begin{enumerate}
\item Let $A$ be an $R$-h-projective dg $R$-algebra, then any object in $hproj-A$ is also $R$-h-projective. Similarly, any object in $hproj-(A^{\op}\otimes_RA)$ is $A$-h-projective, $A^{\op}$-h-projective and $R$-h-projective.

\item Assume in addition that $A$ is compact. Then any object in $Perf (A)$ is $R$-h-projective and $R$-perfect, hence is homotopy equivalent, as an $R$-complex, to a strict perfect complex. The same is true about $R$-complexes $\Hom _A(M,N),\Hom _R(M,R)$ and $M\otimes _AP$ for $M,N\in Perf (A)$ and $P\in Perf (A^{\op})$.

\item If $A$ and $B$ are $R$-h-projective dg $R$-algebras, then so is $A\otimes _RB$. If $R\to T$ is a homomorphism of commutative rings and $A$ is an $R$-h-projective dg $R$-algebra, then $A_T:=A\otimes _RT$ is a $T$-h-projective dg $T$-algebra. If $M$ is an $A$-h-projective module then $M\otimes_RT$ is an $A_T$-h-projective module.

\item Let $M,N$ be $(A^{\op}\otimes_RA)$-modules. Let $R\to T$ be a homomorphism of commutative rings. Denote $(-)_T:=(-)\stackrel{\bbL}{\otimes }_RT$. Then one has a
quasi-isomorphism of $(A^{\op}\otimes_RA)_T$-modules
$$(M\stackrel{\bbL}{\otimes }_A N)_T\simeq M_T\stackrel{\bbL}{\otimes }_{A_T}N_T.$$
\end{enumerate}
\end{lemma}

\begin{proof} (1) For the first statement, take any $M\in hproj-A$, let $C$ be an acyclic $R$-complex.
We have an isomorphism of dg modules
\begin{equation}
\label{eq_MC}
\Hom_R(M,C)\simeq\Hom_A(M,\Hom_R(A,C)).
\end{equation}
Since $A$ is $R$-h-projective, the dg module $\Hom_R(A,C)$ is acyclic. Now since $M$ is $A$-h-projective, the right-hand side of \eqref{eq_MC}  is an acyclic dg module. Thus $M$ is $R$-h-projective. Other statements are proved similarly.

(2) This is clear, as every object in $Perf (A)$ is homotopy equivalent to a direct summand of a dg $A$-module that has a finite filtration with subquotients being shifts of $A$.

(3) Assume that $A$ and $B$ are $R$-h-projective and let $C$ be an acyclic $R$-complex. Then
we have an isomorphism of $R$-complexes
$$\Hom _R(A\otimes _RB,C)=\Hom _R(A,\Hom _R(B,C)),$$
where the complex $\Hom _R(B,C)$ is acyclic (as $B$ is $R$-h-projective), hence also $\Hom _R(A,\Hom _R(B,C))$ is acyclic (as $A$ is $R$-h-projective).

To prove the second statement let $D$ be an acyclic $T$-complex. Then $\Hom _T(A\otimes _RT,D)=\Hom _R(A,D)$ and the second complex is acyclic, because $A$ is $R$-h-projective.

For the last assertion, take any acyclic $A_T$-module $C$. Then there is an isomorphism 
$\Hom_{A_T}(M\otimes_RT,C)\simeq \Hom_A(M,C)$ of complexes. The second complex is acyclic, because $M$ is $A$-h-projective.

(4) Replacing $M$ and $N$ by quasi-isomorphic bimodules we may assume that $M,N\in hproj-(A^{\op}\otimes_R A)$. Then by (1) $M,N$ are also h-projective as $A^{\op},A$ and $R$-modules. Therefore $M_T=M\otimes_RT$ and $N_T=N\otimes_RT$. By (3), $M_T$ is $A_T$-h-projective, thus $M_T\stackrel{\bbL}{\otimes }_{A_T}N_T=(M\otimes_RT)\otimes_{A_T}(N\otimes_RT)$.
We have $M\stackrel{\bbL}{\otimes }_AN=M\otimes_AN$, moreover, this module is $R$-h-projective.
Hence    $(M\stackrel{\bbL}{\otimes }_AN)_T=(M\otimes_AN)\otimes_RT$. Now the statement follows from the standard isomorphism
$$(M\otimes_AN)\otimes_RT\simeq (M\otimes_RT)\otimes_{A\otimes_RT}(N\otimes_RT)$$
of $A_T$-bimodules.
\end{proof}

\begin{definition} \label{mor-eq-r}Dg $R$-algebras $A$ and $B$ (which are $R$-h-projective) are \emph{Morita equivalent} if there exists $K\in D(A^{\op}\otimes _RB)$ such that the functor
$$\Phi _K\colon D(A)\to D(B),\quad M\mapsto M\stackrel{\bbL}{\otimes }_AK$$ is an equivalence.
\end{definition}

\begin{remark}
In the above definition the functor $\Phi_K$ restricts to an equivalence between 
$\Perf(A)$ and $\Perf(B)$ (because $\Perf(A)$ is the subcategory of compact objects in $D(A)$). In particular, $M\otimes_A K\in \Perf(B)$ for any $M\in\Perf(A)$.
\end{remark}

\begin{remark} \label{morita} Morita equivalence is indeed an equivalence relation on the collection of ($R$-h-projective) dg $R$-algebras. Indeed, it is clearly reflexive and transitive. It is also symmetric, as the inverse of the equivalence $\Phi _K\colon D(A)\to D(B)$ is given by the functor $\Phi _{K^t}\colon D(B)\to D(A)$, where $K^t:=\bbR \Hom _B(K,B)\in D(B^{\op}\otimes A)$ (see Lemma \ref{existence-of-kernel}).
\end{remark}

\begin{remark} \label{criterion-for-mor-eq} Dg algebras $A$ and $B$ are Morita equivalent if, for example, they are quasi-isomorphic or, more generally, if the dg categories $Perf (A)$ and $Perf (B)$ are quasi-equivalent, see  \cite{Ke}.
\end{remark}

Let us recall some (surely well known) facts about smooth and compact dg $R$-algebras and the corresponding categories $\Perf(A)$.

\begin{lemma}\label{prop-on-propert-of-smooth-dg-R-algebras}
\begin{enumerate}
\item  Let $A$ and $B$ be smooth ($R$-h-projective) dg $R$-algebras. Then the dg $R$-algebras  $A\otimes _RB$
and $A^{\op}$ are also smooth.

\item If $A$ is a smooth dg $R$-algebra and $R\to T$ is a homomorphism of commutative rings, then $A\stackrel{\bbL}{\otimes }_RT$ is a smooth dg $T$-algebra. Also compactness of a dg algebra is preserved under extension of scalars.
\end{enumerate}
\end{lemma}

\begin{proof} (1) The bifunctor
$$\stackrel{\bbL}{\otimes }_R\colon D(A^{\op}\otimes _RA)\times D(B^{\op}\otimes _RB)\to D((A\otimes _RB)^{\op}\otimes _RA\otimes _RB)$$
maps $\Perf(A^{\op}\otimes _RA)\times \Perf(B^{\op}\otimes _RB)\to \Perf((A\otimes _RB)^{\op}\otimes _RA\otimes _RB)$ and sends the pair of diagonal bimodules $(A,B)$ to the diagonal bimodule $A\otimes _RB$. It follows that $A\otimes _RB$ is smooth.

To see that $A^{\op}$ is smooth, notice the natural isomorphism of dg algebras
$$A^{\op}\otimes _RA\simeq A\otimes _RA^{\op},\ a\otimes b\mapsto b\otimes a.$$
It induces an equivalence of derived categories $D(A^{\op}\otimes _RA)\simeq D(A\otimes _RA^{\op})$ which preserves perfect dg modules and diagonal bimodules.

(2) This is \cite[Thm. 3.30]{LS1}.
\end{proof}

\begin{lemma}
\label{existence-of-kernel} 
Let $A$ and $B$ be $R$-h-projective dg $R$-algebras and let $K\in D(A^{\op}\otimes _RB)$ be an h-projective dg $(A^{\op}\otimes _RB)$-module. Then the following holds.

\begin{enumerate}
\item $K$ is h-projective as a dg $A^{\op}$- or $B$-module.

\item  The functor
$$\Phi _K\colon D(A)\to D(B),\quad M\mapsto M\stackrel{\bbL}{\otimes }_AK=M\otimes _AK$$
has the right adjoint
$$\Psi _K\colon D(B)\to D(A),\quad N\mapsto \bbR \Hom _B(K,N)=\Hom _B(K,N).$$
Let $\alpha \colon Id _{D(A)}\to \Psi _K\cdot \Phi _K$ and $\beta \colon \Phi _K\cdot \Psi _K\to Id_{D(B)}$ be the adjunction morphisms. Then $\alpha (A)\colon A\to \Hom _B(K,K)$ is a morphism of dg $A^{\op}\otimes _RA$-modules (not just dg $A$-modules) and $\beta (B)\colon \Hom _B(K,B)\otimes _AK\to B$ is a morphism of dg $B^{\op}\otimes _R B$-modules (not just dg $B$-modules).

\item  Assume that $K$ is perfect as a dg $B$-module. Then

\begin{enumerate}
\item  The functor $\Phi _K$  (and $\Psi _K$) is an equivalence if and only if the morphisms of dg bimodules $\alpha (A)$ and $\beta (B)$ are quasi-isomorphisms. If it is the case, $\Psi_K$ is the inverse to $\Phi_K$.

\item  The functor $\Psi _K$ is isomorphic to the functor $$\Phi _{K^t}\colon D(B)\to D(A)\quad, N\mapsto N\stackrel{\bbL}{\otimes }_BK^t,$$
where $K^t:=\Hom _B(K,B)$.

\item  Let $R\to T$ be a homomorphism of commutative rings. Put $(-)_T:=(-)\stackrel{\bbL}{\otimes}_RT$.
Then the dg $B^{\op}_T\otimes A_T$-modules $(K^t)_T$ and $(K_T)^t$ are quasi-isomorphic.


\item  If $\Phi _K$ is an equivalence, so is $\Phi _{K_T}$.
\end{enumerate}

\item  Assume that $\Phi _K$ is an equivalence. Then 
$$\Phi^{\op}_{K^t}=K^t\otimes_A(-)\colon D(A^{\op})\to D(B^{\op})$$ is also the equivalence. Moreover, the functor
\begin{equation}\label{second-funct}\Phi _{K^t\boxtimes K}\colon D(A^{\op}\otimes _RA)\to D(B^{\op}\otimes _RB),\ \ S\mapsto K^t\stackrel{\bbL}{\otimes }_AS\stackrel{\bbL}{\otimes }_AK
\end{equation}
is an equivalence which preserves the diagonal bimodules and the subcategories $\Perf$.

\item  If $A$ and $B$ are Morita equivalent, then $A$ is smooth if and only if $B$ is smooth.
\end{enumerate}
\end{lemma}

\begin{proof} (1)  Let $C$ be an acyclic dg $B$-module. Then we have a natural isomorphism of $R$-complexes
$$\Hom _B(K,C)=\Hom _{A^{\op}\otimes_RB}(K,\Hom_B(A\otimes_RB,C))=
\Hom _{A^{\op}\otimes_RB}(K,\Hom_R(A,C)),$$
and the latter complex is acyclic, since $A$ is $R$-h-projective and $K$ is $A^{\op}\otimes_RB$-projective. Similarly one proves that $K$ is h-projective as a dg $A^{\op}$-module.

(2) The adjunction statement follows from the functorial isomorphism of complexes
$$\Hom _B(M\otimes _AK,N)\simeq \Hom _A(M,\Hom _B(K,N)).$$
The last assertions follows from the explicit formulas for the maps $\alpha (A)$ and $\beta (B)$:
$$\alpha (A)\colon A\to \Hom (K,K),\quad a\mapsto a\cdot id _K,\quad \text{and}\quad \beta (B)\colon \Hom _B(K,B)\otimes _AK\to B,\quad f\otimes k\mapsto f(k).$$

(3) a) The ``only if'' direction is clear. Assume that $\alpha (A)$ and $\beta (B)$ are quasi-isomorphisms.
Notice that the functor $\Phi _K$ preserves direct sums and the functor $\Psi _K$ preserves direct sums, since $K_B$ is perfect. It follows that morphisms of functors $\alpha $ and $\beta$ are isomorphisms on triangulated subcategories which contain $A$ and $B$ respectively and are closed under arbitrary direct sums. But then they are isomorphisms on the whole categories $D(A)$ and $D(B)$.

b) Let $N\in hproj-B$. Consider the morphism of functors
$\gamma \colon \Phi _{K^t}\to \Psi _K$
$$\gamma (N)\colon N\otimes _B\Hom _B(K,B)\to \Hom _B(K,N),\quad n\otimes f\to (k\mapsto nf(k)).$$
Then $\gamma (B)$ is an isomorphism of complexes. Also both functors $\Phi _{K^t}$ and $\Psi _K$ preserve direct sums. It follows that $\gamma$ is an isomorphism as in the proof of a) above.

c) It suffices to prove that the dg $B_T^{\op}\otimes _TA_T$-modules $\Hom _B(K,B)_T$ and $\Hom _{B_T}(K_T,B_T)$ are quasi-isomorphic. We have the natural morphism of dg $B_T^{\op}\otimes _TA_T$-modules
$$\delta \colon \Hom _B(K,B)_T\to \Hom _{B_T}(K_T,B_T),\quad f\otimes t\to (k\otimes t'\mapsto f(k)tt').$$
This map is an isomorphism (of dg $B_T^{\op}$-modules) if $K=B$. Hence it is a quasi-isomorphism for all $K$ which are perfect as dg $B$-modules.

d) By part (3) a) we know that $\Phi _K$ is an equivalence if and only if the natural morphisms
$$\alpha (A)\colon A\to \Hom _B(K,K),\quad \text{and} \quad \beta (B)\colon \Hom _B(K,B)\otimes _AK\to B$$
are quasi-isomorphisms. So let us assume that these are quasi-isomorphisms.

By Lemma \ref{perf-r-compl} the $R$-complexes $\Hom _B(K,K)$ and $\Hom _B(K,B)$ are $R$-h-projective. (Indeed,  $A$ and $B$ are $R$-h-projective, $K$ is a perfect dg $B$-module which is also $B$-h-projective by part (1); hence $K$ is $R$-h-projective). Since $K$ is h-projective as a dg $A^{\op}$-module (by part (1)), it follows that the $R$-complex $\Hom _B(K,B)\otimes _{A}K$ is also h-projective. Therefore the morphisms
\begin{equation}\label{isomorph} \alpha (A)_T\colon A_T\to \Hom _B(K,K)_T,\quad \text{and} \quad \beta (B)_T\colon (\Hom _B(K,B)\otimes _AK)_T\to B_T
\end{equation}
are also quasi-isomorphisms. Notice that the corresponding morphisms
$$\alpha (A_T)\colon A_T\to \Hom _{B_T}(K_T,K_T),\quad \text{and} \quad \beta (B_T)\colon \Hom _{B_T}(K_T,B_T)\otimes _{A_T}K_T\to B_T$$
satisfy $\alpha (A_T)=\phi (K)\cdot \alpha (A)_T$ and $\beta (B_T)\cdot \psi (K)=\beta (B)_T$ where $\phi (K)$ and $\psi(K)$ are the obvious maps
\begin{equation}\label{isom-2} \phi (K)\colon \Hom _B(K,K)_T\to \Hom _{B_T}(K_T,K_T)
\end{equation}
and
\begin{equation}\label{isom-3}
\psi (K)\colon (\Hom _B(K,B)\otimes _AK)_T\to \Hom _{B_T}(K_T,B_T)\otimes _{A_T}K_T.
\end{equation}
One shows that $\phi (K)$ and $\psi (K)$ are quasi-isomorphisms as in the proof of part c) above. This implies that $\alpha (A_T)$ and $\beta (B_T)$ are quasi-isomorphisms, i.e. $\Phi _{K_T}$ is an equivalence.

(4) Let $\Phi _K$ be an equivalence. The $\Phi _K(A)=K\in \Perf(B)$. Hence $\Psi_K\simeq \Phi_{K^t}$ by (3) b). By (3) a) the adjunction morphisms
\begin{equation}
\label{eq_alphabeta}
\alpha(A)\colon A\to K\otimes_BK^t\quad\text{and}\quad \beta(B)\colon K^t\otimes_AK\to B
\end{equation} 
are quasi-isomorphisms. Note that $(K^t)^t=K$ and $K^t\in\Perf(B^{\op})$, hence the same conditions \eqref{eq_alphabeta} imply that the functor $\Phi^{\op}_{K^t}=K^t\otimes_A(-)\colon \Perf(A^{\op})\to\Perf(B^{\op})$ is an equivalence by (3) a).

To prove that $\Phi _{K^t\boxtimes K}$ is an equivalence, note that $K^t\boxtimes K\in\Perf(B^{\op}\otimes_RB)$ and use (2) and (3) a),b). The right adjoint to $\Phi _{K^t\boxtimes K}$
is $\Phi _{K\boxtimes K^t}$. It suffices to check that the homomorphism
$$\alpha(A\otimes_R A)\colon A\otimes_R A\to \Phi_{K\boxtimes K^t}\Phi_{K^t\boxtimes K}(A\otimes_R A)=\Phi_{K\boxtimes K^t}(K^t\boxtimes K)=(K\otimes_BK^t)\boxtimes (K\otimes_BK^t)$$
is a quasi-isomorphism, and similarly for $\beta$. But this is true since
$\alpha(A\otimes_RA)=\alpha(A)\boxtimes \alpha(A)$ and $\alpha(A)$ is a quasi-isomorphism of $R$-h-projective modules. Finally, by (2) and quasi-isomorphisms~\eqref{eq_alphabeta} 
the functor $\Phi_{K^t\boxtimes K}$ preserves diagonal bimodules.

(5) This follows from (4).
\end{proof}


If a dg $R$-algebra $A$ is compact by Lemma  \ref{perf-r-compl} we obtain a dg bi-functor
$$\Hom _A(-,-)\colon Perf (A)^{\op}\times Perf (A)\to Perf (R)$$
which induces a bi-triangulated functor
$$\bbR \Hom _A(-,-)\colon \Perf(A)^{\op}\times \Perf(A)\to \Perf(R).$$
Denote by $(-)^*$ the anti-involution $\Perf(R)^{\op}\stackrel{\sim}{\to}\Perf(R)$, where $P^*=\bbR \Hom _R(P,R)$.

\begin{predl} \label{serre-general-over-r}Let $A$ be a ($R$-h-projective) smooth and compact dg $R$-algebra.

\begin{enumerate}
\item A dg $A$-module $M$ is perfect if and only if it is perfect as an $R$-complex.

\item  The category $\Perf(A)$ has a Serre functor. That is there exists a triangulated auto-equivalence
$S\colon \Perf(A)\to \Perf(A)$ and an isomorphism of functors from $\Perf(A)^{\op}\times \Perf(A)\to \Perf(R)$
$$\bbR \Hom _ A(M,N)^*\simeq \bbR \Hom _A(N,S(M)).$$
Moreover the functor $S$ is isomorphic to $(-)\stackrel{\bbL }{\otimes }_AA^*$, where $A^*:=\Hom _R(A,R)\in D(A^{\op}\otimes A)$. The dg $A^{\op}\otimes _RA$-module $A^*$ is perfect. We call it the \emph{Serre bimodule}.

\item  The Serre bimodule and the Serre functor are preserved by extension of scalars. That is if $R\to T$ is a homomorphism of commutative rings and $A_T=A\otimes _RT$ then there is a quasi-isomorphism of dg $A_T$-bimodules $(A_T)^*\simeq (A^*)\stackrel{\bbL}{\otimes }_RT$.

Now assume that the ring $R$ is Noetherian.

\item  For any $M,N\in \Perf(A)$ the $R$-module $\Hom ^{\bul}_{\Perf(A)}(M,N)$ is finitely generated.

\item  Let $\text{mod-}R$ denote the abelian category of finitely generated $R$-modules. Then any cohomological functor $F\colon \Perf(A)^{\op}\to \text{mod-}R$ is representable. That is there exists $M_F\in \Perf(A)$ and an isomorphism of functors $F(-)\simeq \Hom _{\Perf(A)}(-, M_F)$. 
\end{enumerate}

\end{predl}

\begin{proof}
(1) It is clear that a perfect dg $A$-module is also perfect as an $R$-complex (Lemma \ref{perf-r-compl} (2)).
Vice versa, let $M$ be a dg $A$-module which is perfect as an $R$-complex. The functor $$(-)\otimes _AA\colon \Perf(A)\to \Perf(A)$$
is isomorphic to the identity functor. But $A$ is a perfect dg  $A^{\op}\otimes _RA$-module, hence it suffices to show that $M\otimes _A(A^{\op}\otimes _RA)$ is a perfect dg $A$-module. We have
$$M\otimes _A(A^{\op}\otimes _RA)=M\otimes _RA$$
which is in $\Perf(A)$ since $M\in \Perf(R)$.

(2) Our proof is an adaptation of the corresponding proof in \cite{Sh} for the case $R=k$. Let $M,N\in \Perf(A)$. We need to show the existence of a functorial isomorphism in $\Perf(R)$:
\begin{equation}\label{isom-serre-funct}
\bbR \Hom _A(M,N)^*\simeq \bbR \Hom _A(N,M\stackrel{\bbL}{\otimes}_AA^*).
\end{equation}
This follows from the sequence of functorial isomorphisms of $R$-complexes for $M,N\in Perf (A)$:
\begin{equation} \label{details-for-isom}
\begin{array}{rcl}
\Hom _R(\Hom _A(M,N),R) & = & \Hom _R(N\otimes _A\Hom _A(M,A),R)\\
 & = & \Hom _A(N,\Hom _R(\Hom _A(M,A), R))\\
 & = & \Hom _A(N,M\otimes _A\Hom _R(A,R)).
 \end{array}
\end{equation}
To see that $A^*\in \Perf(A^{\op}\otimes _RA)$ we may use the fact that $A^*\in \Perf(R)$ (Lemma \ref{perf-r-compl}), part (1) in this proposition and Proposition \ref{prop-on-propert-of-smooth-dg-R-algebras}(1).




(3) By Lemma \ref{perf-r-compl} the $R$-complex $A^*$ is h-projective, hence
$$A^*\stackrel{\bbL}{\otimes }_RT=A^*\otimes _RT=\Hom _R(A,R)\otimes _RT.$$
For the same reason we have
$$(A_T)^*=\Hom _T(A\otimes _RT,T).$$
Now the statement follows from the isomorphism of dg $A_T$-modules
$$\alpha \colon \Hom _R(A,R)\otimes _RT\to \Hom _T(A\otimes _RT,T),\quad
\alpha (f\otimes t)(a\otimes t')=f(a)tt'.$$

(4) It suffices to note that $\Hom_{\Perf(A)}^{\bul}(A,A)=\oplus _iH^i(A)$ is a finitely generated $R$-module, since $A$ is assumed to be a perfect $R$-complex.

(5) This follows from (4) and \cite[Cor. 4.17]{Rou}, which is a direct generalization of the corresponding theorem \cite[Thm 1.3]{BVdB} treating  the case when $R$ is field.
\end{proof}



\subsection{Semi-orthogonal decompositions}
Let $T$ be a triangulated category. Recall that two full triangulated subcategories $T_1,T_2\subset T$ form a semi-orthogonal decomposition $T=\langle T_1,T_2\rangle$, if $\Hom (T_2,T_1)=0$ and for every object $X\in T$ there exists a distinguished triangle in $T$:
$$X_2\to X\to X_1\to X_2[1]$$
with $X_i\in T_i$.  One can show that for each $X$ the triangle as above is unique up to an isomorphism and the association $X\mapsto X_1$ (resp. $X\mapsto X_2$) induces a triangulated functor $p_1\colon T\to T_1$ (resp. $p_2\colon T\to T_2$) which is the left (resp. right) adjoint to the natural embedding $T_1\hookrightarrow T$ (resp. $T_2\hookrightarrow T$). The functors $p_i$ are Verdier localizations.

\begin{corollary}\label{cor-for-gen-semi-orth} Let $T$ be a triangulated category with a semi-orthogonal decomposition $T=\langle T_1,T_2\rangle$.
If $E\in T$ is a generator, then $E_i=p_i(E)$ is a generator of $T_i$ for $i=1,2$. Moreover, if $\langle E\rangle _n=T$, then similarly $\langle E_i\rangle _n=T_i$, $i=1,2$. Vice versa, if $E_i$ is a generator for $T_i$, $i=1,2$, then $E_1\oplus E_2$ is a generator for $T$.
\end{corollary}

\begin{proof} Indeed, since the functors $p_i$ are Verdier localizations the first assertion follows from Remark  \ref{rem-gen-loc}. The last assertion follows from the definition of a semi-orthogonal decomposition.
\end{proof}

\begin{lemma} \label{lemma-semi-orth} Let $\k$ be a field. Let $C$ be a dg $\k$-algebra and let $\Perf(C)=\langle T_1,T_2\rangle$ be a semi-orthogonal decomposition. Let $\cT _1 ,\cT _2 \subset Perf (C)$ be the induced enhancements of $T_1$ and $T_2$ respectively (see Remark \ref{induced-enhancement}).  Then there exist dg algebras $A$ and $B$ and a dg $B^{\op}\otimes _kA$-module~$M$ such that

\begin{enumerate}
\item the dg algebra $C$ is Morita equivalent to the triangular dg algebra
$$\tilde{C}=\begin{pmatrix}
B & M\\
0 & A
\end{pmatrix};$$

\item there are quasi-equivalences of pre-triangulated categories
$\cT_1\simeq Perf (A)$ and $\cT_2\simeq Perf (B)$; hence in particular one has natural equivalences of triangulated categories $T_1\simeq \Perf(A)$ and $T_2\simeq \Perf(B)$.
\end{enumerate}
\end{lemma}

\begin{proof} Choose h-projective generators $E_1$ and $E_2$ of $T_1$ and $T_2$ respectively. Then by Corollary \ref{cor-for-gen-semi-orth} $E:=E_1\oplus E_2$ is a generator for $\Perf C$. By Lemma \ref{ess-lemma} the dg category $Perf (C)$ is quasi-equivalent to the dg category $P erf (\cE)$ where $\cE$ is the dg algebra $\End _{P erf (C)}(E)$. Hence the dg algebras $C$ and~$\cE$ are Morita equivalent (Remark \ref{criterion-for-mor-eq}). If $A=\End _{\cT _1}(E_1)$, $B=\End _{\cT _2}(E_2)$ then the dg algebra $\cE$ has the matrix form
$$
\begin{pmatrix}
B & {}_BM_A\\
{}_AN_B & A
\end{pmatrix},
$$
where $M=\Hom _{P erf (C)}(E_1,E_2)$ and
$N=\Hom _{P erf (C)}(E_2,E_1)$. By the assumption the complex
$N$ is acyclic. Hence the dg algebra $\cE$ is quasi-isomorphic to its dg subalgebra
$$\tilde{C}=\begin{pmatrix}
B & M\\
0 & A
\end{pmatrix}.
$$
Thus in particular $\cE$ and $\tilde{C}$ are Morita equivalent (Remark \ref{criterion-for-mor-eq}).
Therefore the dg algebra $C$ is Morita equivalent to the dg algebra $\tilde{C}$, which proves the first assertion. The second one follows again from Lemma \ref{ess-lemma}.
\end{proof}

\subsection{Krull-Schmidt categories}\label{subs-krull-schmidt} 
For basic facts about Krull-Schmidt categories we refer to \cite[(2.2)]{Ri}. 
Recall that an additive category is called \emph{Krull-Schmidt} if every object decomposes into a finite direct sum of indecomposable objects and endomorphism rings of indecomposable rings are local. In a Krull-Schmidt category a decomposition into indecomposable summands is  unique up to isomorphisms of summands. It is known that any $\k$-linear idempotent complete category with finite dimensional $\Hom$-spaces is Krull-Schmidt. In particular, we have the following.
 Let $A$ be a compact dg algebra over a field. Then the category $\Perf(A)$ is Krull-Schmidt.

\section{Rouquier dimension}
\label{section_rouquier}
Let $T$ be a triangulated category.

Following \cite{Rou} we make the following definition.

\begin{definition} Let $n$ be the least integer such that there exists an object $E$ in $T$ with $T=\langle E\rangle _n$. Then we say that the \emph{Rouquier dimension} of $T$, denoted $\Rdim T$, is $n$. If no such $n$ exists, then the Rouquier dimension is $\infty$. Thus $\Rdim T<\infty$ if and only if $T$ has a strong generator.
\end{definition}

The following results were proved in \cite{Rou} (7.9, 7.16, 7.17, 7.25, 7.37, 7.38).

\begin{theorem}\label{thm-rouq} 
\begin{enumerate}
\item Let $X$ be a separated scheme of finite type over a perfect field. Then the category $D^b(\coh X)$ has finite Rouquier dimension.

\item Let $X$ be a reduced separated scheme of finite type over a field. Then $\Rdim  D^b(\coh X)\ge \dim X$.

\item Let $X$ be a smooth affine scheme over a field. Then $\Rdim D^b(\coh X)=\dim X$.

\item Let $X$ be a smooth quasi-projective scheme over a field. Then
$\Rdim D^b(\coh X)\le 2\dim X$.

\item Let $A$ be an Artinian ring. Then $\Rdim D^b(A\text{-mod})$ is bounded above by 
the Loewy length of $A$, i.e. by the minimal $d$ such that $(\mathrm{rad}(A))^{d+1}=0$.

\item Let $A$ be a Noetherian ring. Then $\Rdim \Perf(A)<\infty$ if and only if $\gldim  A<\infty$.
\end{enumerate}
\end{theorem}

The next proposition should be known to experts, but we did not find a reference.

\begin{predl}
Let $A$ be a Noetherian ring of global dimension $n$. Then
$$\Rdim \Perf(A)\le n.$$
\end{predl}
\begin{proof}
We claim that $\Perf(A)=\langle A\rangle _n$. Let $C_{\bul}$ be a bounded complex of finitely generated $A$-modules. Take its Cartan-Eilenberg resolution  $K_{\bul\bul}$. It is a bounded bicomplex of projective finitely generated $A$-modules having the following properties:
\begin{enumerate}
\item total complex $Tot(K_{\bul\bul})$ is a projective resolution of $C_{\bul}$.
\item horizontal differentials are split homomorphisms;
\item number of rows in $K_{\bul\bul}$ is $\le n+1$.
\end{enumerate}
Clearly, $C$ is obtained from rows of $K$ by $n$ cones. On the other hand, any row of $K$ is homotopy equivalent to the direct sum of its cohomology modules, which are projective finitely generated modules. Hence every row of $K$ belongs to $\langle A\rangle _0$.
\end{proof}

It follows from Theorem \ref{thm-rouq} that for $X$ smooth and quasi-projective over a field the following holds: 
$$\dim  X\le \Rdim D^b(\coh X)\le 2\dim X.$$ 
Rouquier remarks that in all cases where $\Rdim  D^b(\coh X)$ can be computed exactly, one actually has $\Rdim  D^b(\coh X)=\dim  X$. In \cite{Or} Orlov conjectures that this equality holds for every smooth quasi-projective variety $X$. He
also proves the following

\begin{theorem} Let $X$ be a smooth projective curve over a field. Then $\Rdim  D^b(\coh X)=1$.
\end{theorem}

\subsection{$\Rdim =0$}

\begin{lemma}\label{crit-for-rdim-0} Assume that $T$ is a  triangulated Krull-Schmidt category. For example $T$ can be $\Perf(A)$ for a compact dg algebra $A$ over a field (see Subsection \ref{subs-krull-schmidt}). Then $\Rdim T=0$ if and only if $T$ contains only finitely many indecomposables up to an isomorphism and a shift.
\end{lemma}

\begin{proof}  Let $E_1,\ldots ,E_t$ be all indecomposables (up to isomorphism and shift) in $T$. Put $E=\oplus E_i$. Then $T=\langle E\rangle _0$. Vice versa, let $E\in T$ be such that $\langle E\rangle _0=T$. Let $E=\oplus E_i$ be the (finite) decomposition of $E$ into indecomposables. Then clearly $\{E_i\}$ is the complete list of indecomposables in $T$ up to an isomorphism and shift.
\end{proof}

\begin{example}\label{ex-rou-dim-zero}
Assume that $\k$ is an algebraically closed field. Let $Q$ be a finite quiver and $A=k[Q]$ be its path algebra. Then the category $D^b(\text{mod-}A)=\Perf(A)$ has Rouquier dimension equal to zero if and only if $Q$ is a Dynkin quiver of A-D-E type. Indeed, by a theorem of Gabriel \cite{Ga} the abelian category $\text{mod-}A$
contains only finitely many indecomposables (up to isomorphism) if and only if $Q$ is a Dynkin quiver of A-D-E type. Notice that the algebra $A$ has global dimension $1$ and so any complex in
$D^b(\text{mod-}A)$ is a direct sum of its cohomology. So the category $D^b(\text{mod-}A)$ contains finitely many indecomposables (up to isomorphism and shift) if and only if $Q$ is a Dynkin quiver of A-D-E type. It remains to apply Lemma \ref{crit-for-rdim-0}.
\end{example}

Essentially, the above example is the only possible example of a category with Rouquier dimension zero. One has the following result by N.\,Hanihara \cite[Theorems 1.3, 1.4]{Han}, see also Amiot's work \cite{Am}.
\begin{theorem}
Let $T\simeq\Perf(A)$ for a compact dg algebra $A$ over an algebraically closed field.  Suppose $\Rdim T=0$. Then $T\simeq D^b(\modd \k[Q])$, where $Q$ is a finite disjoint union of quivers of A-D-E type. 
\end{theorem}
\begin{remark}
The cited result uses that $T$ has infinitely many isomorphism classes of indecomposable objects. In our case this condition is satisfied because all shifts $E[i], i\in\Z$ of  a single indecomposable object $E$ are non-isomorphic (since $\Hom^{\bul}(E,E)$ is finite-dimensional). 
\end{remark}

\subsection{\bf  Monotonicity}\label{subs-monot} Let $T$ be a triangulated category and let $l\colon T\to T^\prime$ be its Verdier localization. Then $\Rdim  T\ge \Rdim  T^\prime $. Indeed, the functor $l$ is essentially surjective, hence if $T=\langle E\rangle _n$ for an object $E\in T$, then $T^\prime =\langle l(E)\rangle _n$.

If $T^\prime $ is a semi-orthogonal component of $T$, then the inclusion functor $T^\prime \hookrightarrow T$ has an adjoint
$T\to T^\prime$ which is a localization. Hence $\Rdim  T\ge \Rdim  T^\prime $.

\subsection{\bf Additivity} Assume that $A$ and $B$ are dg algebras over a field. Consider the triangulated categories $\Perf(A)$, $\Perf(B)$ and $\Perf(A\otimes B)$. Then the sum $\Rdim  \Perf(A) + \Rdim  \Perf(B)$ may not be equal to $\Rdim  \Perf(A\otimes B)$, see Example \ref{example_maintable}. However, we don't know of an example where $\Rdim  \Perf(A\otimes B)<\Rdim  \Perf(A)+\Rdim \Perf(B)$.

\subsection{Behavior in families}\label{rou-dim-in-fam} We expect that the Rouquier dimension is upper semi-continuous. Namely, let $R$ be a commutative Noetherian ring and let $A$ be a smooth and compact dg $R$-algebra. For $x\in \Spec R$, denote by 
$A_x$  the dg $k(x)$-algebra $A\stackrel{\bbL}{\otimes }_Rk(x)$, where $k(x)$ is the residue field of the point $x$.
Then the function
$\Rdim  \Perf(A_x)$, $x\in \Spec R$, should be upper semi-continuous on $\Spec R$. 

%


\section{Diagonal dimension}
\label{section_diagonal}


We consider triangulated categories of the form $\Perf(A)$, where $A$ is a dg $\k$-algebra. Consider the bifunctor
$$\boxtimes \colon \Perf(A^{\op})\times \Perf(A)\to \Perf(A^{\op}\otimes  A).$$

\begin{definition} \label{defi-of-diag-dim} 
Define the \emph{diagonal dimension} of the category $\Perf(A)$ as the minimal $n\in \bbZ _{\ge 0}$ for which there exist objects $F\in \Perf(A^{\op})$ and $G\in \Perf(A)$ such that the diagonal dg $A^{\op}\otimes A$-module $A\in\langle F\boxtimes G\rangle _n\subset \Perf(A^{\op}\otimes A)$. If no such $n$ exists then we say that the diagonal dimension is~$\infty$. We denote the diagonal dimension by $\Ddim  \Perf(A)$.
\end{definition}

Note that the diagonal dimension of $\Perf(A)$ is finite if and only if $A$ is smooth.

\begin{remark} We show below that the diagonal dimension of the category $\Perf(A)$ is a Morita invariant of the dg algebra $A$. But we don't know if it is an invariant of the triangulated category $\Perf(A)$.
\end{remark}

\begin{predl} \label{morita-inv-of-diag} Let $A$ and $B$ be Morita equivalent dg algebras.
Then 
$$\Ddim  \Perf(A)=\Ddim  \Perf(B).$$
\end{predl}

\begin{proof} This follows from Lemma \ref{existence-of-kernel}.
Indeed, let $\Phi_K \colon D(A) \to D(B)$ be an equivalence. Then by Lemma~\ref{existence-of-kernel}(4) the functor
$\Phi_{K^t\boxtimes K}\colon D(A^{\op}\otimes  A) \to D(B^{\op}\otimes  B)$
is an equivalence which preserves the diagonal bimodule. Assume that
$A \in \langle F\boxtimes G\rangle_n $ for some $F\in\Perf(A^{\op})$,$G\in\Perf(A)$ and $n$.
By Lemma~\ref{existence-of-kernel}(4) we have
$$B=\Phi_{K^t\boxtimes K}(A) \in\langle(K^t\otimes_AF)\boxtimes 
(G\otimes_AK)\rangle_n.$$
Note that $K^t\otimes_AF\in \Perf(B^{\op})$ and $G\otimes_AK\in\Perf(B)$ because functors $\Phi^{\op}_{K^t}$ and $\Phi_K$ restrict to equivalences on subcategories $\Perf$.
Hence, $\Ddim \Perf(B)\le \Ddim \Perf(A)$, the opposite inequality is proved similarly.
\end{proof}

We make the following
\begin{remark}
Let $A$ be a dg algebra, assume that $A\in\langle F\boxtimes G\rangle_n$ for some $F\in\Perf(A^{\op})$, $G\in\Perf(A)$. Then for any object $M\in\Perf(A^{\op}\otimes A)$ one has $M\in\langle F'\boxtimes G'\rangle_n$ for some $F'\in\Perf(A^{\op})$, $G'\in\Perf(A)$.

Indeed, for some $B\in \Perf(A^{\op}\otimes A)$ we have $A\oplus B\in [F\boxtimes G]_n$. Tensoring with $M$ we get
$$M\oplus (M\stackrel{\bbL}{\otimes}_AB)\in [(M\stackrel{\bbL}{\otimes}_A F)\boxtimes G]_n,$$
hence one can take $F'=M\stackrel{\bbL}{\otimes}_A F\in\Perf(A^{\op})$ and $G'=G$.
\end{remark}

\subsection{Compatibility with the geometric definition.}\label{geom-def-of-diag-dim}

In the geometric context the notion of the diagonal dimension was introduced in \cite[Definition 2.15]{BaFa}. Namely, for a smooth projective variety $X$ the authors define the diagonal dimension
$\Ddim  D^b(\coh X)$ of the derived category $D^b(\coh X)$ to be the least integer $n$ such that there exist $F,G\in D^b(\coh X)$ so that
$\cO _\Delta \in \langle F\boxtimes G\rangle _n$, where $\cO _\Delta$ is the structure sheaf of the diagonal on $X\times X$.

The triangulated category $D^b(\coh X)$ has an enhancement and is equivalent to the category $\Perf(A)$ for a (smooth and compact) dg algebra $A$. In fact any two enhancements of $D^b(\coh X)$ are quasi-equivalent, see  \cite{LO}. So we may compare the number $\Ddim  D^b(\coh X)$ with $\Ddim  \Perf(A)$ as in Definition \ref{defi-of-diag-dim}. We claim that the two notions of the diagonal dimension agree.

\begin{lemma} In the notation above $\Ddim  D^b(\coh X)=\Ddim  \Perf(A)$.
\end{lemma}

\begin{proof} In  \cite[Thm. 1.1, Sect. 4.4]{LS2} there were constructed equivalences $D^b(\coh X)\simeq \Perf(A)$, $D^b(\coh X)\simeq \Perf(A^{\op})$ and $\theta \colon D^b(\coh (X\times X))\stackrel{\sim}{\to} \Perf(A^{\op}\otimes A)$ which are compatible with the box-product $\boxtimes$ and such that $\theta$ maps the structure sheaf of the diagonal to the diagonal dg bimodule $A$. This implies the assertion of the lemma.
\end{proof}

\subsection{$\Ddim =0$	}

\begin{predl}\label{prop-diag-dim-zero} Let $A$ be smooth and compact dg algebra over a field $\k$. Assume that $\Ddim\Perf(A)=0$. Then $A$ is Morita equivalent to the product
$D_1\times \ldots\times D_s$, where each $D_i$ is a finite dimensional division $\k$-algebra
(concentrated in degree zero).
\end{predl}

\begin{proof}
Assume that the diagonal bimodule $A$ is isomorphic in $\Perf(A^{\op}\otimes A)$ to a direct sum $A=T_1\oplus T_2$ for some $T_1,T_2\in \Perf(A^{\op}\otimes A)$. Then for every $M\in \Perf(A)$ we have
\begin{equation}
\label{eq_MMM}
M\simeq (M\stackrel{\bbL}{\otimes }_AT_1)\oplus (M\stackrel{\bbL}{\otimes }_AT_2).
\end{equation}
Denote by $\cA_1$ (resp. $\cA_2$) the full subcategory in $\Perf(A)$ of such objects $M$ that $M\stackrel{\bbL}{\otimes }_AT_2=0$ (resp. $M\stackrel{\bbL}{\otimes }_AT_1=0$). Clearly, $\cA_1,\cA_2\subset \Perf(A)$ are triangulated subcategories.
Assume that $M\in\cA_1$  and $N\in \cA_2$. Then $\Hom^{\bul}(M,N)=\Hom^{\bul}(N,M)=0$. Indeed, if $f\in \Hom (M,N[i])$, then $f\stackrel{\bbL}{\otimes }_AT_1=0=f\stackrel{\bbL}{\otimes }_AT_2$, hence $f=0$. 
If $M\in\Perf(A)$ is an indecomposable object then by \eqref{eq_MMM} either $M\in\cA_1$ or $M\in\cA_2$.
Recall that the category $\Perf(A)$ is Krull-Schmidt (see Subsection~\ref{subs-krull-schmidt}), it follows that for any object $M\in\Perf(A)$ 
we have a decomposition $M\simeq M_1\oplus M_2$ where $M_1\in\cA_1, M_2\in\cA_2$.
Therefore, one has an orthogonal decomposition $\Perf(A)=\cA _1\times \cA _2$.
By Lemma \ref{lemma-semi-orth} the dg algebra $A$ is Morita equivalent to a product of dg algebras $A_1\times A_2$ and by Proposition \ref{monot-for-diag-dim} we have $\Ddim \Perf(A_i)=0$, $i=0,1$. So we may and will assume that the diagonal bimodule $A$ is indecomposable in $\Perf(A^{\op}\otimes A)$.

By assumption $A\oplus T\simeq F\boxtimes G$ for some $T\in \Perf(A^{\op}\otimes A)$, $F\in \Perf(A^{\op})$ and $G\in \Perf(A)$. It follows that any $M\in \Perf(A)$ is a direct summand of $(M\stackrel{\bbL}{\otimes} _AF)\otimes _kG$. Because the bimodule $A$ is indecomposable and the category $\Perf(A^{\op}\otimes A)$ is Krull-Schmidt, we may and will assume that $G$ is indecomposable. Now by Krull-Schmidt property of $\Perf(A)$ we get that $\Perf(A)=[G]_0$. It remains to apply the following lemma.
\end{proof}

\begin{lemma} Let $A$ be a compact dg $\k$-algebra such that $\Perf(A)=[G]_0$ for some $G\in \Perf(A)$. Then $A$ is Morita equivalent to a finite dimensional division $\k$-algebra.
\end{lemma}

\begin{proof} Since $G$ is a generator of $\Perf(A)$, the dg algebra $A$ is Morita equivalent to the endomorphism dg algebra of $G$ (Corollary \ref{cor-quasi-equiv}).
Hence we may and will assume that $G=A$. It suffices to prove that $H^i(A)=0$ if $i\neq 0$ and that $H^0(A)$ is a division algebra.

Choose $0\neq f\in H^i (A)$ and consider it as a morphism $f\colon A\to A[i]$.  The cone of $f$ is a direct sum of shifts of $A$. We need to check that $i=0$ and $f$ is an isomorphism in $\Perf A$. Assume the contrary: $f$ is not an isomorphism in $\Perf(A)$ (i.e. $f\colon H(A)\to H(A)$ is not an isomorphism). By considering the dimension of the cohomology we conclude that $Cone (f)\simeq A[j]$ for some $j$. 

We get an exact triangle $A\xra{f} A[i]\to A[j]\to A[1]$. Looking at the associated long exact sequence in cohomology one sees that this triangle can be either 
$$A\xra{f} A\to A\to A[1], \quad A\xra{f} A\to A[1]\to A[1], \quad\text{or}\quad A\xra{f} A[1]\to A[1]\to A[1].$$ 
Further, any map in the long exact sequence in cohomology is either an isomorphism or zero. 

Assume $i=0$. If the map $f\colon H^0(A)\to H^0(A)$ is an isomorphism then  $f$ is a unit in $H(A)$, hence $f\colon A\to A$ is an isomorphism in $\Perf(A)$, a contradiction.
If the map  $f\colon H^0(A)\to H^0(A)$ is the zero map then  $f=0\in H^0(A)$, a contradiction. Now assume $i=1$, then  we have an exact triangle 
$A\xra{f} A[1]\xra{g} A[1]\to A[1]$. As was shown above, the map $g$ is either zero or an isomorphism. 
In each case we obtain a contradiction.

Hence, $f$ is an isomorphism $A\to A[i]$ in $\Perf A$. It follows immediately
that $i=0$ since $A$ is compact.
%
%
\end{proof}

\subsection{Additivity}

\begin{predl} \label{add-for-diag-dim} Let $A$ and $B$ be dg algebras. Then
$$\Ddim  \Perf(A\otimes B)\le \Ddim \Perf(A)+\Ddim \Perf(B).$$
\end{predl}

\begin{proof} We may assume that the diagonal dimensions of $\Perf(A)$ and $\Perf(B)$ are finite. So there exist $F\in \Perf(A)$, $G\in \Perf(A^{\op})$,
$F^\prime \in \Perf(B)$, $G^\prime \in \Perf(B^{\op})$ such that
$A\in \langle G\otimes F \rangle _n$ and $B\in \langle G^\prime \otimes F^\prime \rangle _m$. We claim that in this case
$$A\otimes B\in \langle G\otimes F\otimes G^\prime \otimes F^\prime \rangle _{n+m}\subset \Perf(A^{\op}\otimes A\otimes B^{\op}\otimes B)=\Perf((A\otimes B)^{\op}\otimes (A\otimes B)).$$
This follows from the simple general lemma.

\begin{lemma} Let $C$ and $D$ be dg algebras, $K\in \Perf(C)$, $L\in \Perf(D)$. Then for any $n$ and $m$ we have the following inclusion of subsets of objects in $\Perf(C\otimes D)$:
$$
\langle K\rangle _n\otimes \langle L\rangle _m\subset \langle K\otimes L\rangle _{n+m}.
$$
\end{lemma}

\begin{proof} It suffices to prove the inclusion
\begin{equation*}
[K] _n\otimes [L] _m\subset [K\otimes L] _{n+m}.
\end{equation*}
Let $P\in [K]_n$. Replacing $P$ by an isomorphic object if necessary we may assume that
there exists a filtration $P=P_n\supset P_{n-1}\supset \ldots \supset P_0$ by dg submodules and a collection of objects $Q_0,\ldots ,Q_n\in [K]_0$ such that $P_0=Q_0$ and for each $i>0$, $P_i=Cone(Q_i[-1]\stackrel{f_i}{\to} P_{i-1})$ for a (closed degree zero) morphism $f_i$. (Hence the quotient dg module $P_i/P_{i-1}$ is isomorphic to $Q_i$). Similarly every object in $[L]_m$ has an isomorphic object $P^\prime$ with a filtration
$P^\prime =P^\prime _m\supset P^\prime _{n-1}\supset \ldots \supset P^\prime _0$ such that there exist objects $Q^\prime _0,\ldots ,Q^\prime _m\in [L]_0$ and isomorphisms $P^\prime _0=Q^\prime _0$ and $P^\prime _i=Cone (Q^\prime _i[-1]\stackrel{f^\prime _i}{\to}P^\prime _{i-1})$.

Then the object $S:=P\otimes P^\prime$ has the filtration $S=S_{n+m}\supset \ldots \supset S_0$, where $S_a:=\sum _{i+j\le a}P_i\otimes P^\prime _j$. Moreover, for each $a>0$ the dg module $S_a/S_{a-1}$ is
isomorphic to the direct sum $\bigoplus _{i+j=a}Q_i\otimes Q^\prime _j$. Hence $S_a$ is isomorphic to the cone of a morphism $g_a \colon \bigoplus _{i+j=a}Q_i\otimes Q^\prime _j[-1]\to S_{a-1}$. It follows
that $S_a\in [K\otimes L]_a$, in particular $S\in [K\otimes L]_{n+m}$. This proves the lemma.
\end{proof}
The proposition now follows if we take $C=A^{\op}\otimes A$, $D=B^{\op}\otimes B$, $K=G\otimes F$, $L=G'\otimes F'$.
\end{proof}


\subsection{$\Ddim \ge \Rdim$.}

\begin{predl}\label{ddim-geq-rdim} Let $A$ be a dg algebra. Then
$$\Rdim  \Perf(A)\le \Ddim  \Perf(A).$$
\end{predl}

\begin{proof} We may assume that $\Ddim  \Perf(A)=n <\infty$.
Let $F \in \Perf(A^{\op})$ and $G\in \Perf(A)$ be such that
$A\in \langle F\boxtimes G \rangle _n$, i.e. there exists $M\in \Perf(A^{\op}\otimes A)$ such that
\begin{equation}
\label{ddim-cond} A\oplus M\in [F\boxtimes G]_n.
\end{equation}
Consider the functor
$$
\Phi _{A\oplus M}\colon D(A)\to D(A), \quad N\mapsto N\stackrel{\bbL}{\otimes }_A(A\oplus M)\simeq N\oplus (N\stackrel{\bbL}{\otimes}_AM).
$$
Then \eqref{ddim-cond} implies that
every $N\in D(A)$ is contained in the subcategory $\langle G\rangle ^{\oplus}_n$, where the category $\langle G\rangle ^{\oplus}_n$ is defined the same way as $\langle G\rangle _n$ but using arbitrary direct sums instead of finite ones. It follows from \cite[Prop. 2.2.4]{BVdB} that if $N\in \Perf(A)(=D(A)^c)$ belongs to $\langle G\rangle ^{\oplus}_n$ then $N\in \langle G\rangle _n$. That is $\Rdim  \Perf(A)\le n$.
\end{proof}

\begin{example} \label{ex-strict} Let $\k$ have characteristic $p>0$ and let $L/k$ be an inseparable field extension of degree $p$. Then $\Rdim  \Perf(L)=0$. On the other hand $L\otimes _kL\simeq L[\epsilon]/(\epsilon ^p)$ and hence $L\notin \Perf(L^{\op}\otimes L)$. Therefore the dg $\k$-algebra $L$ is not smooth and $\Ddim  \Perf(L)=\infty$.
\end{example}

\subsection{Monotonicity}

Assume that $C$ is a dg algebra and let $\Perf(C)=\langle T,T^\prime \rangle$ be a semi-orthogonal decomposition. As explained in Remark \ref{induced-enhancement} the categories $T$ and $T^\prime$ have  canonical enhancements $\cT $ and $\cT ^\prime$ induced by the enhancement $Perf (C)$ of $\Perf(C)$. By Lemma \ref{lemma-semi-orth} the dg categories  $\cT$ and $\cT ^\prime$
are respectively quasi-equivalent to dg categories $Perf(A)$ and $Perf (B)$ for certain dg algebras  $A$ and $B$ (so that $T\simeq \Perf(A)$ and $T^\prime \simeq \Perf(B)$). In this context we have the following result.

\begin{predl} \label{monot-for-diag-dim} In the above notation $\Ddim  \Perf(C)\ge \Ddim  \Perf(A), \Ddim  \Perf(B)$.
\end{predl}

\begin{proof} We may assume that $\Ddim  \Perf(C)<\infty$.

By Lemma \ref{lemma-semi-orth} and Proposition \ref{morita-inv-of-diag} we may assume that $C$ is the triangular dg algebra
$$C=\left(\begin{array}{cc}
B & M\\
0 & A
\end{array}
\right)
$$
where $A$ and $B$ are dg algebras and $M$ is a dg $B^{\op}\otimes A$-module.
A (right) dg $C$-module can be described as a row vector $[Y,X]$ where $Y$ is a dg $B$-module and $X$ is a dg $A$-module.
Let
$$e_A=\left(\begin{array}{cc}
0 & 0\\
0 & 1
\end{array}\right)\quad \text{and} \quad e_B=\left(\begin{array}{cc}
1 & 0\\
0 & 0
\end{array}\right)$$
be the corresponding idempotents in $C$.

We have the obvious homomorphism of dg algebras
$$C\to A, \quad \left(\begin{array}{cc}
b & m\\
0 & a
\end{array}\right)\mapsto a$$
which induces the extension of scalars: a dg functor
$$\psi _A\colon Perf (C)\to Perf (A),\quad M\mapsto Me_A, \quad [Y,X]\mapsto X.$$
Similarly we get the dg functors
$$\psi _{A^{\op}}\colon Perf (C^{\op})\to Perf (A^{\op}),\quad N\mapsto e_AN$$
and
$$\psi _{A^{\op}\otimes A}\colon Perf (C^{\op}\otimes C)\to Perf (A^{\op}\otimes A),\quad S\mapsto e_ASe_A.$$
Clearly, for $M\in Perf (C),N\in Perf (C^{\op})$ we have $\psi _{A^{\op}}(N)\otimes \psi _A(M)=\psi_{A^{\op}\otimes A}(N\otimes M)$. Also $\psi _{A^{\op}\otimes A}$ maps the diagonal dg bimodule $C$ to the diagonal dg bimodule $A$. It follows that $\Ddim  \Perf(C)\ge \Ddim  \Perf(A)$. Similarly one proves that $\Ddim  \Perf(C)\ge \Ddim  \Perf(B)$.
\end{proof}

\subsection{Behavior of $\Ddim $ in families} We have no results in this direction but expect, like in the case of the Rouquier dimension (\ref{rou-dim-in-fam}), that the diagonal dimension is upper semi-continuous in families.

\subsection{$\Ddim$ for a category with a full exceptional collection}

Diagonal dimension of a triangulated category with a full exceptional collection is bounded by the length of the collection. To be more precise, we have

\begin{predl}
\label{prop_ddimexcoll}
Suppose that $A$ is a smooth dg algebra over a field $\k$ and assume that $\Perf(A)$ has a full exceptional collection having the block structure
$$\Perf(A)=\left\langle \begin{matrix}E_{0,1}, & E_{1,1}, & \ldots, & E_{n,1}\\
 & & \ldots & \\
E_{0,d_0}, & E_{1,d_1}, & \ldots, & E_{n,d_n}
\end{matrix}
\right\rangle.$$
Then 
$$\Ddim\Perf(A)\le n.$$
\end{predl}
\begin{proof}
For any object $E_{ij}$ consider the functor
$$\phi_{ij}=-\boxtimes E_{ij}\colon \Perf(A^{\op})\to \Perf(A^{\op}\otimes A).$$
It is fully faithful, denote its image by $\BB_{ij}$. Then the subcategories $\BB_{ij}$ form a semi-orthogonal decomposition 
$$\Perf(A^{\op}\otimes A)=\left\langle \begin{matrix}\BB_{0,1}, & \BB_{1,1}, & \ldots, & \BB_{n,1}\\
 & & \ldots & \\
\BB_{0,d_0}, & \BB_{1,d_1}, & \ldots, & \BB_{n,d_n}
\end{matrix}
\right\rangle.$$
Denote $\BB_k:=\oplus_{i=1}^{d_k} \BB_{ki}$, then 
$$\Perf(A^{\op}\otimes A)=\langle \BB_0, \ldots, \BB_n\rangle. $$
Since $A\in\Perf(A^{\op}\otimes A) $ it follows that there exists a sequence 
$$0=F_{n+1}\to F_{n}\to \ldots \to F_1\to F_0=A$$
in $\Perf(A^{\op}\otimes A)$
such that $Cone(F_{k+1}\to F_k)\in \BB_k$ for any $k=0,\ldots,n$.
Suppose that $Cone(F_{k+1}\to F_k)=\oplus_{i=1}^{d_k}B_{ki}\boxtimes E_{ki}$, then clearly $A\in \langle (\oplus_{k,i} B_{ki})\boxtimes (\oplus_{ki} E_{ki})\rangle_n$. This proves the statement.
\end{proof}

\subsection{$\Ddim  D^b(\coh X)$ for a quasi-projective $X$}

Recall (see Theorem \ref{thm-rouq}(2),(4)) that for a smooth quasi-projective $\k$-scheme $X$ we have the estimate
$$\dim X\le \Rdim  D^b(\coh X) \le 2\dim X.$$
A similar estimate holds for the diagonal dimension.

\begin{lemma} Let $X$ be a smooth quasi-projective scheme over a perfect field $\k$. Then
$$\dim X\le \Ddim  D^b(\coh X)\le 2\dim X.$$
\end{lemma}

\begin{proof} The first inequality follows from Theorem \ref{thm-rouq}(2) and Proposition \ref{ddim-geq-rdim}. To prove the second we remark that the product $X\times X$ is smooth of dimension $2\dim X$. Note that since $X$ is quasi-projective,  every coherent sheaf on $X\times X$ is a quotient of a sheaf $M\boxtimes N$ for some locally free sheaves $M$ and $N$ on $X$.
Therefore we can find a resolution of the structure sheaf of the diagonal on $X\times X$
$$0\to K\to M_{2\dim X}\boxtimes N_{2\dim X}\to \ldots \to M_0\boxtimes N_0\to \cO _{\Delta} \to 0$$
The complex $0\to M_{2\dim X}\boxtimes N_{2\dim X}\to \ldots \to M_0\boxtimes N_0\to 0$ defines an element in $\Ext ^{2\dim X +1}(\cO _\Delta ,K)$. But this group is zero, since $\Ext ^{2\dim X +1}(S,T)=0$ for any coherent sheaves $S,T$ on the smooth variety $X\times X$ . It follows that $\cO _{\Delta}$ is a direct summand of the complex $M_{2\dim X}\boxtimes N_{2\dim X}\to \ldots \to M_0\boxtimes N_0$. Therefore $\cO _\Delta \in \langle M\boxtimes N\rangle _{2\dim X}$, where $M=\oplus M_i$ and $N=\oplus N_i$. So $\Ddim  D^b(\coh X)\le 2\dim X$ (see Section~\ref{geom-def-of-diag-dim}).
\end{proof}

In some special cases we have the equality $\Ddim D^b(\coh X)=\dim X$. By Proposition~\ref{prop_ddimexcoll}, this equality holds for any projective variety $X$ with a full exceptional collection in $D^b(\coh X)$ consisting of $n=\dim X$ blocks.
For example, we have $\Ddim D^b(\coh X)=\dim X$ for $X$ being a projective space, a smooth quadric, a smooth del Pezzo surface or a Fano threefold $V_5$ or $V_{22}$.

Moreover, let $X$ be a smooth affine variety over a perfect field. Then $\Ddim D^b(\coh X)=\dim X$. This follows from the next 
\begin{predl}
Let $A$ be a finite-dimensional or a commutative and essentially finitely generated algebra 
  over a perfect field $\k$. Then $\Ddim D^b(\modd A)\le \gldim A$.
\end{predl}
\begin{proof}
Note that the algebra $A^{\op}\otimes A$ is Noetherian.
Denote $n=\gldim A$. We demonstrate that $A\in \langle A\otimes A\rangle_n$. 
Indeed, by \cite[Lemma 7.2]{Rou} the projective dimension of the $A^{\op}\otimes A$-module $A$ is $n$. Since any finitely generated projective $A^{\op}\otimes A$-module is in $\langle A\otimes A\rangle_0$, the statement follows.
\end{proof}

But in general, we expect that for most smooth and projective varieties $X$ the diagonal dimension is bigger than $\dim X$. In particular we make the following

\begin{conjecture} Let $X$ be a smooth and projective irrational curve over $\k$. Then $\Ddim  D^b(\coh X)=2$.
\end{conjecture}

\section{Serre dimension}
\label{section_serre}

Let $\k$ be a field. We keep the same notational conventions as in Section \ref{section_diagonal}.

\subsection{Definition of the Serre dimension for categories $\Perf(A)$}

Recall the definition of the Serre functor \cite{BK1}.

\begin{definition} Let $T$ be a $\k$-linear triangulated category which is Ext-finite, i.e. for any objects $X,Y\in T$, $\sum _i\dim \Hom _{T}(X,Y[i])<\infty$. An endofunctor $S\colon T \to T$ is the \emph{Serre functor} if for any $X,Y\in T$ there is a functorial isomorphism of vector spaces
$$\Hom (X,Y) = \Hom (Y,S(X))^*$$
where $(-)^*$ denotes the dual vector space.
\end{definition}

The Serre functor, if it exists, is unique up to an isomorphism  and is triangulated. Theorem \cite[Thm. 1.3]{BVdB} implies in particular that the Serre functor exists for an Ext-finite triangulated category $T$ which is Karoubian and has a strong generator. Therefore the Serre functor exists for the category $\Perf(A)$ where $A$ is a smooth and compact dg algebra, or where $A$ is a finite dimensional algebra which has finite global dimension. By Proposition \ref{serre-general-over-r} the Serre functor is given by the following formula:
$$S(X)=X\stackrel{\bbL}{\otimes }_AA^*$$
where $A^*:=\Hom _k(A,k)\in \Perf(A^{\op}\otimes A)$.
 We denote the $m$-fold tensor product 
 $$A^*\stackrel{\bbL}{\otimes }_A \ldots \stackrel{\bbL}{\otimes }_AA^*$$ by $(A^*)^{\otimes^{\bbL}_Am}$.

Given a triangulated category $T$ and an endofunctor $F\colon T\to T$ one can define the upper and lower $F$-dimension of $T$, denoted $F\Udim T$ and $F\Ldim T$. This is done in Section \ref{section_general-dim} below.
In particular one can define the upper and lower Serre dimension of a category that has a Serre functor. Here we only specialize to the case when $T$ has a generator.

\begin{definition} 
\label{definition_e-e+}
For two objects $G_1,G_2$ in a triangulated category we denote
$$e_-(G_1,G_2)=\inf\{i\mid \Hom^i(G_1,G_2)\ne 0\}; \quad
e_+(G_1,G_2)=\sup\{i\mid \Hom^i(G_1,G_2)\ne 0\}.$$
If $\Hom^{\bul}(G_1,G_2)=0$ we put $e_-(G_1,G_2)=+\infty, e_+(G_1,G_2)=-\infty$.
Also for a complex $C^\bullet$ we define
$$\inf C^{\bul}=\inf\{i\mid H^i(C^{\bul})\ne 0\}, \qquad
\sup C^{\bul}=\sup\{i\mid H^i(C^{\bul})\ne 0\}.$$
\end{definition}

\begin{definition} 
\label{definition_sdim}
Let $T$ be a $\k$-linear Ext-finite triangulated category with a Serre functor $S\colon T\to T$. Assume that $T$ has a generator. Choose any generators $G,G'$ of $T$ and define the \emph{upper Serre} and the \emph{lower Serre} dimension of $T$ as follows
$$\USdim T=\limsup_{m\to +\infty} \frac{-e_-(G,S^m(G'))}m,\quad
  \LSdim T=\liminf_{m\to +\infty} \frac{-e_+(G,S^m(G'))}m.
$$
\end{definition}

The fact that the numbers $\USdim  T$ and $\LSdim T$ are well defined (i.e. don't depend on the choice of generators) follows from Lemma~\ref{lemma_GG} and Definition~\ref{definition_Fdim2}. It is clear that equivalent triangulated categories have equal upper Serre and lower Serre dimensions.
\begin{remark}
In Definition~\ref{definition_sdim}, if $T$ is a non-zero category then for any $G,G'$ and $m$ we get $\Hom^{\bul}(G,S^m(G'))\ne 0$. Consequently, $e_-\le e_+$ are finite numbers and one has 
$$\LSdim T\le \USdim T.$$
For $T=0$ Definition~\ref{definition_sdim} gives $\LSdim T=+\infty$, $\USdim T=-\infty$, this does not look fine. Further when dealing with the Serre dimension we will always assume that the category is non-zero.
\end{remark}

In case $A$ is a smooth and compact dg $\k$-algebra and $T=\Perf(A)$, one can replace $\limsup$ and $\liminf$ in the above definition by $\lim$.

\begin{predl} \label{form-for-serre-dim} Let $A$ be a smooth and compact dg algebra over $\k$. Choose any generators $G,G'$ of $\Perf(A)$. Then
$$\USdim \Perf(A)=\lim_{m\to +\infty} \frac{-e_-(G,S^m(G'))}m,\quad
  \LSdim \Perf(A)=\lim_{m\to +\infty} \frac{-e_+(G, S^m(G'))}m.
$$
In particular, if we put $G=G'=A$ the above formulas read as follows
\begin{equation}
\label{eq_sdimA}
\USdim  \Perf(A)=\lim_{m\to +\infty}\frac{-\inf(A^*)^{\otimes^{\bbL}_Am}}m,\quad
\LSdim  \Perf(A)=\lim_{m\to +\infty}\frac{-\sup(A^*)^{\otimes^{\bbL}_Am}}m.
\end{equation}

Moreover these limits are finite.
\end{predl}

\begin{proof} This follows immediately from Propositions~\ref{prop_entropy-to-dimension} and~\ref{prop_Fdimfinite}.
\end{proof}

\begin{lemma} 
\label{lemma_sdimgeom}
Suppose that $X$ is a smooth  projective equidimensional variety over $\k$. Then
$$\USdim  D^b(\coh X)=\LSdim  D^b(\coh X)=\dim X.$$
\end{lemma}

\begin{proof} Let $n=\dim X$. Recall that the Serre functor on $D^b(\coh X)$ is isomorphic to $(-)\otimes \omega _X[n]$. Recall also that for any coherent sheaves $F_1,F_2$ on $X$ we have $\Ext ^i(F_1,F_2)=0$ if $i\notin [0,n]$. 
Choose a coherent sheaf $G$ such that $G$ is a generator of $D^b(\coh X)$. Then  we have
$$-e_-(G,S^m(G))=-e_-(G,G\otimes \omega ^{\otimes m}_X[mn])\le mn$$
and
$$-e_+(G,S^m(G))=-e_+(G,G\otimes \omega ^{\otimes m}_X[mn])\ge mn-n.$$
Since $-e_-(G,S^m(G))\ge -e_+(G,S^m(G))$ it follows from Proposition \ref{form-for-serre-dim}
that 
$$\USdim  D^b(\coh X)=\LSdim  D^b(\coh X)=n.$$
\end{proof}

There exist examples of categories with negative Serre dimension.

\begin{example}
\label{neg-serre}
Let $A=k\langle \varepsilon \rangle$ be the algebra of dual numbers. Consider it as a  dg algebra with zero differential and $\deg \varepsilon =w\in\Z$. The category $\Perf(A)$ has the Serre functor $S(-)=(-)\stackrel{\bbL}{\otimes }_AA^*$, where $A^*=\Hom _k(A,k)\simeq A[w]$. Thus the Serre functor
is isomorphic to the shift $[w]$. So $\LSdim  \Perf(A)=\USdim  \Perf(A)=w$; it is negative if $w<0$.
\end{example}

Note that the dg algebra $A$ in Example \ref{neg-serre} is compact but not smooth.  

The next example provides a smooth and compact dg algebra with negative lower Serre dimension. It was communicated to us by D.\,Orlov.
\begin{example}[See {\cite[Section 8]{El}}]
\label{example_orlov}
Consider the triangulated category $T$ generated by an exceptional pair $E_1,E_2$. Let $V$ be the graded space $\Hom^{\bul}(E_1,E_2)$, suppose $V$ is finite dimensional. Alternatively, consider the quiver with two vertices and $n=\dim V$ arrows going in one direction. Put gradings on arrows and consider the corresponding graded path algebra $A=\k e_1\oplus \k e_2\oplus V$ with zero differential. Then $T\simeq\Perf(A)$. Such $A$ is smooth and compact. Suppose that $n\ge 2$, then we have $\Rdim T=\Ddim T=1$. Denote by $w=\sup V-\inf V$ the difference between the degrees of the maximal and the minimal nonzero graded components of $V$. Then 
$$\LSdim T=1-w,\quad\USdim T=1+w.$$
In particular, $\LSdim T$ can be negative.
\end{example}

We expect that categories $\Perf(A)$ for smooth and compact dg algebras $A$ have nonnegative upper Serre dimension. We don't know if the Serre dimension can be irrational.

\begin{conjecture} \label{conj-pos-rat} Let $A$ be a smooth and compact dg algebra. Then the Serre dimensions 
$$\LSdim  \Perf(A)\qquad\text{and}\qquad \USdim  \Perf(A)$$ are  rational. Moreover, $\USdim\Perf(A)$ is nonnegative.
\end{conjecture}

\begin{remark} Let $A$ be a finite dimensional algebra of finite global dimension, then  $\LSdim  \Perf(A)\ge 0$. Indeed, replacing the Serre bimodule $A^*$ by its projective resolution $P^\bullet$ we see that for any $m$ the complex $(A^*)^{\otimes ^{\bbL}_Am}=(P^\bullet)^{\otimes _Am}$ lives in non-positive degrees.
\end{remark}

\subsection{$\LSdim =0$ and $\USdim =0$} 

For the lower Serre dimension, there are some nontrivial categories with dimension zero.

\begin{example}[See {\cite[Example 9.4]{El}}]
\label{example_xyz}
Consider the quiver 
$$
\xymatrix{0 \bul  \ar[rr]^x  && \bul 1 \ar[rr]^y && \bul 2 \ar@/^1pc/[llll]_z}
$$
with relations $zy=xz=0$. Denote the corresponding path algebra with relations by $A$ (the field is arbitrary).  Then $A$ is a finite-dimensional algebra with $\gldim A=3$.
One has $\sup (A^*)^{\otimes^{\bbL}_An}=0$ for any $n\ge 0$, therefore $\LSdim \Perf(A)=0$.
\end{example}

For the upper Serre dimension we have the following conjecture.

\begin{conjecture} \label{conj-serre-dim-zero}Let $A$ be a finite dimensional algebra. Assume that $A$ has finite global dimension. Then $\USdim  \Perf(A)=0$ if and only if $\gldim A=0$.
\end{conjecture}

Of course the ``if'' direction of the conjecture is easy. If fact it is a special case of the following general result.

\begin{predl} \label{serre-leq-gldim} Let $A$ be a finite dimensional algebra which has  finite global dimension~$d$. Then $\USdim  \Perf(A)\le d$.
\end{predl}

\begin{proof} Consider the Serre bimodule $A^*$. It suffices to show that it has a resolution
$$0\to P_d\to \ldots \to P_0\to A^*$$
by $A$-bimodules such that each $P_i$ is projective as a right $A$-module. Indeed, then $(A^*)^{\otimes ^{\bbL}_Am}=(P_\bullet)^{\otimes _Am}$ and clearly $-\inf (P_\bullet ^{\otimes m})\le md$. To construct a resolution $P_\bullet$ as above one may start by taking $P_0=A^*\otimes _kA$ with the obvious surjection of bimodules $P_0\to A^*$ and apply the same procedure to its kernel instead of $A^*$. After $d$ steps one obtains as a kernel a bimodule $P_d$ which is projective as a right $A$-module.
\end{proof}

\subsection{Monotonicity} 
\label{section_serremonot}
Monotonicity for Serre dimension does not hold as the following examples demonstrate.

\begin{example}
Let $A$ be the algebra from Example \ref{example_xyz} and $\AA=\Perf(A)$. Let $P_0$ and $P_1$ be the projective modules corresponding to vertices $0$ and $1$. Then $(P_0,P_1)$ is a strong exceptional collection in $\AA$. Hence the subcategory $\BB= \langle P_0,P_1\rangle$ is admissible in $\AA$, it is equivalent to $\Perf(B)$ where $B=\End(P_0\oplus P_1)$ is the path algebra of the quiver $\bul\to\bul$ of type $A_2$. By Example \ref{example_maintable}, one has 
$$\LSdim \BB=1/3>0=\LSdim \AA.$$
\end{example}

We have learned about the next example from D.\,Orlov.
\begin{example}
\label{example_F3}
Let $X=\mathbb F_3$ be the Hirzebruch surface. We claim that the category $D^b(\coh X)$ has an admissible subcategory $\BB$ equivalent to a category from Example~\ref{example_orlov}
with $w=2$. Then 
$$\USdim \BB=3>2=\USdim D^b(\coh X),$$ 
where the last equality is by Lemma~\ref{lemma_sdimgeom}.

Let $F\subset X$ be a fiber and $S\subset X$ be the $(-3)$-curve. Consider the exceptional collection of line bundles
$$(\O_X(-F), \O_X, \O_X(S))$$
on $X$. Let $E$ be the right mutation $R_{\O_X(S)}(\O_X)$. Then the pair $(\O_X(-F),E)$ 
is exceptional and $\Hom^i(\O_X(-F),E)\ne 0$ for $i=-1,0,1$. Consequently in notation of Example~\ref{example_orlov} one has $w=2$. Therefore $\USdim\langle \O_X(-F),E\rangle=3$. 
\end{example}

One might expect that monotonicity of Serre dimension holds for categories such that 
$\LSdim=\USdim$. But, to our knowledge, even the following special case is an open problem: given smooth projective varieties $X$ and $Y$ such that $D^b(\coh X)$ is a semi-orthogonal component of $D^b(\coh Y)$, show that $\dim X\le \dim Y$.

\subsection{Additivity} The additivity holds for the Serre dimension. Recall that if $A$ and $B$ are smooth and compact dg algebras then so is $A\otimes B$ (Lemma \ref{prop-on-propert-of-smooth-dg-R-algebras}).

\begin{predl} \label{addit-for-serre-dim} Let $A$ and $B$ be smooth and compact dg algebras. Put $C=A\otimes B$. Then
$$\USdim  \Perf(C)=\USdim  \Perf(A)+\USdim  \Perf(B)$$
and
$$\LSdim  \Perf(C)=\LSdim  \Perf(A)+\LSdim  \Perf(B).$$
\end{predl}

\begin{proof} Put $C=A\otimes B$. Everything follows from the quasi-isomorphism of complexes for any $m\ge 1$.
$$(C^*)^{\otimes ^{\bbL}_Cm}\simeq (A^*)^{\otimes ^{\bbL}_Am}\otimes (B^*)^{\otimes ^{\bbL}_Bm}.$$
\end{proof}

\subsection{Behavior in families}
\label{section_serreinfamilies}
Here we prove the following result. Recall that a function $f\colon X\to \R$ on a topological space $X$ is called \emph{upper} (resp. \emph{lower}) \emph{semi-continuous} if for any $c\in \R$ the subset $f^{-1}([c,+\infty))$ (resp. $f^{-1}((-\infty,c])$) is closed in $X$.

\begin{theorem} 
\label{theorem_serre-in-families} 
Let $R$ be a commutative Noetherian ring and let $A$ be an $R$-h-projective smooth and compact dg $R$-algebra.
For a point $x\in \Spec R$ denote by $A _x$ the smooth and compact dg $k(x)$-algebra $A \otimes _Rk(x)$. 
\begin{enumerate}
\item[(a)]
Then for any $c\in\R$ the subsets
\begin{equation}\label{closed-subsets} \{x\in \Spec R\mid \USdim  \Perf(A _x)\ge c\}\quad \text{and} \quad \{x\in \Spec R\mid \LSdim  \Perf(A _x)\le c\}
\end{equation}
are closed under specialization in $\Spec R$.

\item[(b)] Assume in addition that $A$ is negatively graded:  i.e. $A^i=0$ for $i>0$. Then the subsets \eqref{closed-subsets} are closed in Zariski topology. That is, $\USdim $ and $\LSdim $ are respectively upper and lower semi-continuous functions on $\Spec R$.
\end{enumerate}
\end{theorem}

We expect that the subsets \eqref{closed-subsets} are closed for all $R$-h-projective smooth and compact dg $R$-algebras.

The proof of Theorem~\ref{theorem_serre-in-families} will take several steps.

\begin{lemma} \label{lemma-loc-rings}Let $B$ be a local Noetherian ring with the residue field $\k$ and let $Q^{-1}\stackrel{d^{-1}}{\to}Q^0\stackrel{d^0}{\to} Q^1$ be a complex of finitely generated free $B$-modules. Then $H^0(Q^\bullet \otimes _Bk)=0$ if and only if the following two conditions hold:

\begin{enumerate}
\item The $B$-modules $Im(d^{-1})$ and $Im(d^0)$ are free;

\item $rk(d^{-1}\otimes _Bk)\ge rkQ^0-rk(Im(d^0))$.
\end{enumerate}
(In fact, conditions (1) and (2) imply that in (2) one has an equality.)
\end{lemma}

\begin{proof} Assume that $H^0(Q^\bullet \otimes _Bk)=0$. Then by \cite[Lemma 2.1.3]{Lieb} there is a submodule $K\subset Q^0$ which maps isomorphically by $d^0$ onto $Im(d^0)$ and such that $Q^0=Im (d^{-1})\oplus K$. This implies (1) and (2).

Vice versa, assume that (1) and (2) hold. Since $Im(d^0)$ is free, there exists a submodule $K\subset Q^0$ which maps isomorphically by $d^0$ onto $Im (d^0)$, i.e.
\begin{equation}\label{ker-oplus}
Q^0=Ker(d^0)\oplus K.
\end{equation}
It suffices to prove that $Im (d^{-1})=Ker (d^0)$. But this follows from the inclusion $Im (d^{-1})\subset Ker(d^0)$ together with \eqref{ker-oplus}, the assumption (2) and the Nakayama lemma.
(Note that this implies  that the inequality in (2) is actually an equality).
\end{proof}

\begin{predl} 
\label{lemma_restriction_inf} 
Let $R$ be a commutative Noetherian ring and let $P^\bullet$ be a perfect $R$-complex. We identify $P^\bullet$ with a
complex of quasi-coherent sheaves on $\Spec R$. Then for any $i$ the subset
$$\{x\in \Spec R \mid H^i(P^\bullet \stackrel{\bbL}{\otimes}_Rk(x))=0\}$$
is open in $\Spec R$. In particular the functions
$$\sup (P^\bullet \stackrel{\bbL}{\otimes}_Rk(x))\quad \text{and}\quad \inf (P^\bullet \stackrel{\bbL}{\otimes}_Rk(x))$$
are upper and lower continuous on $\Spec R$.
\end{predl}

\begin{proof} The second assertion follows from the first one.

To prove the first assertion we may replace $P^{\bullet}$ by a quasi-isomorphic strict perfect complex. We may assume that $i=0$ and that $P^\bullet$ is the complex $P^{-1}\stackrel{d^{-1}}{\to} P^0\stackrel{d^0}{\to} P^1$ of finitely generated free $R$-modules. Consider the $R$-modules $Im (d^{-1})$ and $Im (d^0)$.

For a point $x\in Spec R$ and an $R$-module $M$ denote by $M_x$ its stalk at $x$, it is a module over the local ring $R_x$. We have $H^0(P^\bullet \otimes _Rk(x))=H^0(P^\bullet _x\otimes _{R_x}k(x))$. Therefore we may apply Lemma \ref{lemma-loc-rings} to conclude that $H^0(P^\bullet \otimes _Rk(x))=0$ if and only if

\begin{enumerate}
\item $Im (d^{-1})_x$ and $Im (d^{0})_x$ are free $R_x$-modules and

\item $rk(d^{-1}\otimes _Rk(x))\ge rkP^0_x-rk(Im(d^0)_x)$.
\end{enumerate}

The condition (1) is open in $\Spec R$ by \cite[Ch. II, Exercise 5.7 (a)]{Ha}. Then the open subset of $\Spec R$ where (1) holds is a disjoint union of open subsets where the ranks of the locally free coherent sheaves $P^0$, $Im (d^{-1})$, and $Im (d^0)$ are constant. It suffices to show that the condition (2) is open in any of these open subsets. This follows from the observation that the subset
$$\{ x\in \Spec R\mid rk(d^{-1}\otimes _Rk(x))\ge a\}$$
is open for any $a$, since it is given by nonvanishing of some minors of a matrix with entries in~$R$.
\end{proof}

We can now prove Theorem~\ref{theorem_serre-in-families}.

\begin{proof}
(a) By Proposition \ref{serre-general-over-r}  the Serre bimodule
$$A^*:=\Hom _R(A,R)$$
is perfect as a dg $A^{\op}\otimes _RA$-module, dg $A$-module or an $R$-complex. Similarly $(A^*)^{\stackrel{\bbL}{\otimes }_Am}$ is a perfect $R$-complex.

By Proposition \ref{serre-general-over-r} for any point $x\in \Spec R$ we have
$$A^*\stackrel{\bbL}{\otimes }_Rk(x)=A_x^*:=\Hom _{k(x)}(A_x,k(x))$$
which is the Serre bimodule for $A_x$. 
Also note that by Lemma~\ref{perf-r-compl}(4)
$$((A^*)^{\stackrel{\bbL}{\otimes }_Am})\stackrel{\bbL }{\otimes }_Rk(x)\simeq (A_x^*)^{\stackrel{\bbL}{\otimes} _{A_{x}}m}$$
 for any $x\in \Spec R$ and $m>0$.
Since the $R$-complex $(A^*)^{\stackrel{\bbL}{\otimes }_Am}$ is perfect, it follows from Proposition \ref{lemma_restriction_inf} that if $x,y\in \Spec R$ and $x\in \overline{\{y\}}$, then $H^i((A_{y }^*)^{\otimes _{A_y}m})\neq 0$ implies that $H^i((A_{x }^*)^{\otimes _{A_x}m})\neq 0$. 
Hence, $\inf (A_{y }^*)^{\otimes _{A_y}m}\ge \inf (A_{x }^*)^{\otimes _{A_x}m}$. Passing to limits, we get $\USdim\Perf(A_y)\le \USdim\Perf(A_x)$. Similarly, 
$\LSdim\Perf(A_y)\ge \LSdim\Perf(A_x)$.
This proves the first assertion of the theorem.

(b) Now assume that $A$ is in addition negatively graded. We first prove that the function $\LSdim  \Perf(A_x)$ is lower semi-continuous on $\Spec R$.

Choose an h-projective resolution $P^\bullet \to A^*$ of the dg $A^{\op}\otimes _RA$-module $A^*=\Hom _R(A,R)$. Then $P^\bullet$ is also h-projective as right dg $A$-module and 
$$(A^*)^{\otimes ^{\bbL}_An}=(P^\bullet)^{\otimes _An}$$
is a perfect $R$-complex. By Proposition \ref{serre-general-over-r} for each $x\in \Spec R$ we have $(A_x)^*=(A^*)\otimes _Rk(x)$ and
$$((A_x)^*)^{\otimes ^{\bbL}_{A_x}n}=(P^\bullet )^{\otimes _An}\otimes _Rk(x)$$

Recall (Proposition \ref{form-for-serre-dim}) that
$$\LSdim  \Perf(A_x)=\lim _{n\to \infty}\frac{-\sup ((A_x)^*)^{\otimes ^{\bbL}_{A_x}n}}{n}.$$

Fix $c\in \bbR$. We claim that
$$\{x\in \Spec R\mid \LSdim  \Perf(A_x)>c\}=\bigcup_{m\in\Z, n\in\N\colon \frac{m}{n}>c}
\{x\in \Spec R\mid \sup ((A_x)^*)^{\otimes ^{\bbL}_{A_x}n}\le -m\}.$$
Indeed, suppose for some $x$, $\LSdim  \Perf(A_x)>c$. Then
$$\frac{-\sup ((A_x)^*)^{\otimes ^{\bbL}_{A_x}n}}{n}=\frac{m}{n}>c$$
for some $m,n$. Vice versa, suppose for some $n_0,m_0$ we have
$\sup ((A_x)^*)^{\otimes ^{\bbL}_{A_x}n_0}\le -m_0$ and $m_0/n_0>c$. It follows from Lemma~\ref{lemma_ng} below that for any $l>0$ we have
$$\sup ((A_x)^*)^{\otimes ^{\bbL}_{A_x}ln_0}\le -lm_0$$
(here we need $A$ to be negatively graded!) and so
$$\frac{-\sup ((A_x)^*)^{\otimes ^{\bbL}_{A_x}ln_0}}{ln_0}\ge \frac{m_0}{n_0}.$$
Passing to the limit as $l\to \infty$ we find that $\LSdim  \Perf(A_x)\ge \frac{m_0}{n_0}>c$.

Now we apply Proposition \ref{lemma_restriction_inf} to the perfect $R$-complex $(P^\bullet)^{\otimes _An}$ and conclude that for all $m,n$ the subset
$$\{x\in \Spec R\mid \sup ((A_x)^*)^{\otimes ^{\bbL}_{A_x}n}\le -m\}$$
is open in $\Spec R$. Therefore also the subset $\{x\in \Spec R\mid \LSdim  \Perf(A_x)>c\}$ is open for all $c$. So the function $\LSdim  \Perf(A_x)$ is lower semi-continuous.

To prove that the function $\USdim  \Perf(A_x)$ is upper semi-continuous on $\Spec R$ we need to use some results of Section \ref{section_general-dim} where we consider the upper and lower dimension of a triangulated category~$T$ with respect to an endofunctor $F$. These are denoted by $F \Udim T$ and $F \Ldim T$. If $T=\Perf(A)$ and  $F=S$ is the Serre endofunctor then $S \Udim T=\USdim  T$ and
$S \Ldim T=\LSdim  T$.

Consider the inverse Serre functor $S^{-1}\colon \Perf(A)\to \Perf(A)$. By Lemma \ref{existence-of-kernel} we have
$$S^{-1}(-)=(-)\stackrel{\bbL}{\otimes }_AA^!$$
where $A^!:=\bbR \Hom _A(A^*,A)$, it is a perfect $R$-complex (Lemma \ref{perf-r-compl}). Also for any $x\in \Spec R$ there is an isomorphism of dg $A_x^{\op}\otimes _{k(x)}A_x$-modules
$$A^!\stackrel{\bbL}{\otimes }_Rk(x)=A^!_x$$
and the functor $(-)\stackrel{\bbL}{\otimes }_AA^!_x$ is the inverse $S_x^{-1}$ of the Serre functor $S_x\colon \Perf(A_x)\to \Perf(A_x)$ (Lemma  \ref{existence-of-kernel}). Now the same argument as above for the lower Serre dimension $\LSdim  \Perf(A_x)=S_x \Ldim \Perf(A_x)$ shows that the function $S^{-1}_x \Ldim \Perf(A_x)$ is lower semi-continuous on $\Spec R$. By Corollary \ref{dim-up-and-lo-inverse}  we know that
$$S^{-1}_x \Ldim \Perf(A_x)=-(S_x \Udim \Perf(A_x)).$$
Hence the function $S_x \Udim \Perf(A_x)=\USdim  \Perf(A_x)$ is upper semi-continuous. This completes the proof of Theorem \ref{theorem_serre-in-families}.
\end{proof}

\begin{lemma}
\label{lemma_ng}
Let $A$ be a negatively graded dg $k$-algebra, i.e. $A^i=0$ for $i>0$. Let $M,N$ be  bounded above dg $A$- and $A^{\op}$-modules respectively. Then 
$$\sup (M\otimes^{\bbL}_AN)\le \sup M+\sup N.$$
\end{lemma}
\begin{proof}
Since $A$ is negatively graded, $M$ admits a resolution $P\to M$ by a semi-free dg $A$-module such that $P^i=0$ for $i>\sup M$. Similarly, there is a semi-free resolution 
$Q\to N$ by a dg $A^{\op}$-module with $Q^j=0$ for $j>\sup N$. Then $M\otimes^{\bbL}_AN\cong P\otimes_AQ$ and 
$\sup (P\otimes_AQ)\le \sup M+\sup N$ because $(P\otimes_AQ)^l=0$ for $l>\sup M+\sup N$.
Indeed, $(P\otimes_AQ)^l$ is a quotient of $(P\otimes_kQ)^l=\oplus_{i+j=l}P^i\otimes_k Q^j=0$.
\end{proof}

\section{Dimension of a category with respect to an endofunctor}
\label{section_general-dim}

\subsection{Upper and lower dimension of a category with respect to an endofunctor}

Recall Definition~\ref{definition_e-e+}: 
  for objects $G_1,G_2$ of a triangulated category $T$ we denote
$$e_-(G_1,G_2)=\inf\{i\mid \Hom^i(G_1,G_2)\ne 0\}, \quad
e_+(G_1,G_2)=\sup\{i\mid \Hom^i(G_1,G_2)\ne 0\}.$$

\begin{definition} 
\label{definition_Fdim1} 
Let $T$ be a triangulated category with a triangulated endofunctor  $F\colon T\to T$. For any $G_1,G_2\in T$ put
$$F \Udim(T,G_1,G_2)=\limsup_{m\to +\infty} \frac{-e_-(G_1,F^m(G_2))}m,\quad
 F \Ldim(T,G_1,G_2)=\liminf_{m\to +\infty} \frac{-e_+(G_1,F^m(G_2))}m.
$$
Thus $F \Udim(T,G_1,G_2),  F \Ldim(T,G_1,G_2) \in \bbR \cup \{\pm \infty \}$.

\end{definition}

The following lemma is straightforward.

\begin{lemma}
\label{lemma_remark} 
Take the assumptions of Definition~\ref{definition_Fdim1}.
\begin{enumerate}
\item  We have 
\begin{align*}
F \Udim (T,G,G')&=F \Udim (T,G[d_1],G'[d_2])\qquad\text{and}\\
F \Ldim (T,G,G')&=F \Ldim (T,G[d_1],G'[d_2]) 
\end{align*} 
for any $d_1,d_2\in \bbZ$.

\item Let $G_1\to G_2\to G_3\to G_1[1]$
be a distinguished triangle in $T$. Then for any $G\in T$
\begin{gather*}
F \Udim (T,G,G_2)\le \max \{ F \Udim (T,G,G_1), F \Udim (T,G,G_3)\}\\
F \Ldim (T,G,G_2)\ge \min \{ F \Ldim (T,G,G_1), F \Ldim (T,G,G_3)\}
\end{gather*}
and similarly
\begin{gather*}
F \Udim (T,G_2,G)\le \max \{ F \Udim (T,G_1,G), F \Udim (T,G_3,G)\}\\
F \Ldim (T,G_2,G)\ge \min \{ F \Ldim (T,G_1,G), F \Ldim (T,G_3,G)\}.
\end{gather*}
\end{enumerate}
\end{lemma}


\begin{lemma}
\label{lemma_GG} 
Suppose $T$ has a generator. Then  for any generators $G,G'$ and any objects $G_1,G_2$ one has
\begin{gather*}
\label{first-equal} 
F \Udim (T,G_1,G_2)\le F \Udim (T,G,G'), \\
F \Ldim (T,G_1,G_2)\ge \notag F \Ldim (T,G,G'). 
\end{gather*}

\end{lemma}
\begin{proof}
We will only prove the first inequality. The proof of the second one is similar.

We first prove $F \Udim (T,G_1,G')\le F \Udim (T,G,G')$ for fixed $G,G'$ and arbitrary $G_1$.
Let $G_1\in \langle G\rangle_n$. We may assume that $G_1\in [G]_n$ and will argue by induction in $n$. Let $n=0$, that is  $G_1$ is isomorphic to a finite sum $\oplus  G[d_i]$. Then we may assume that $G_1=G[d]$ and then $F \Udim (T,G_1,G')= F \Udim (T,G,G')$ by Lemma~\ref{lemma_remark}.

For the induction step, choose a triangle $G_1^\prime \to G_1\to G_1^{\prime \prime}\to G_1^\prime [1]$ where
$G_1\in [G]_n$, $G_1^\prime \in [G]_0$,$G_1^{\prime \prime}\in [G]_{n-1}$. By induction hypothesis,
$F \Udim (T,G_1^\prime,G')\le F \Udim (T,G,G')$ and
$F \Udim (T,G_1^{\prime \prime},G')\le F \Udim (T,G,G')$. Hence also $F \Udim (T,G_1,G')\le F \Udim (T,G,G')$ by Lemma~\ref{lemma_remark}.

Similarly one checks that $F \Udim (T,G_1,G_2)\le F \Udim (T,G_1,G')$ by fixing $G_1,G'$ and varying~$G_2$.
\end{proof}

\begin{definition}
\label{definition_Fdim2}
Let $T$ be a  triangulated category with a generator. Let $F\colon T\to T$ be a triangulated functor. We define the \emph{upper} and \emph{lower $F$-dimension} of $T$ as
$$F\Udim T:=F\Udim(T,G,G'),\quad F\Ldim T:=F\Ldim(T,G,G'),$$
where $G,G'$ are some generators of $T$.
\end{definition}

By Lemma~\ref{lemma_GG}, the quantities $F\Ldim T$ and $F\Udim T$ are well defined: $F$-dimension does not depend on the choice of generators.

\begin{lemma}
Let $T$ be a  triangulated category with a generator. Let $F\colon T\to T$ be a triangulated functor, suppose that $F$ is not nilpotent.
Then
$$F\Ldim T\le F\Udim T.$$
\end{lemma}
\begin{proof}
Indeed, let $G\in T$ be a generator. Then for any $m$ we have $\Hom^{\bul}(G,F^m(G))\ne 0$. Consequently, $e_-(G,F^m(G))\le e_+(G,F^m(G))$. The statement now follows by passing to the limits.
\end{proof}

\begin{predl}
\label{prop_e-e+}
Suppose that $T$ is a $\k$-linear Ext-finite triangulated category with a generator. Suppose also that $T$ has a Serre functor $S$. Let $F\colon T\to T$ be a triangulated functor. Then for any generators $G,G'\in T$ one has
$$F\Ldim T=\liminf_{m\to +\infty} \frac{e_-(F^m(G'),G)}m.$$
\end{predl}
\begin{proof}
Clearly, for any objects $G_1,G_2\in T$ one has 
$$e_-(G_1,G_2)=-e_+(G_2,S(G_1)).$$
Therefore one has
\begin{multline*}
\liminf_{m\to +\infty}\frac{e_-(F^m(G'),G)}m=
\liminf_{m\to +\infty}\frac{-e_+(G,S(F^m(G')))}m=\\
=\liminf_{m\to +\infty}\frac{-e_+(S^{-1}(G),F^m(G'))}m=
F\Ldim T,
\end{multline*}
because $S^{-1}(G)$ is also a generator in $T$.
\end{proof}

\begin{predl} 
\label{dim-up-and-lo-inverse} 
Suppose that $T$ is a $\k$-linear Ext-finite triangulated category with a generator. Suppose also that $T$ has a Serre functor $S$. 
Let $F\colon T \to T$ be an autoequivalence. Then
$$F\Ldim T=-(F^{-1})\Udim T.$$
\end{predl}

\begin{proof} 
Indeed, let $G,G'\in T$ be generators. Then by Proposition~\ref{prop_e-e+} we have
\begin{multline*}
F\Ldim T  = 
\liminf_{m\to +\infty} \frac{e_-(F^m(G^\prime),G)}m =
\liminf_{m\to +\infty} \frac{e_-(G^\prime,F^{-m}(G))}m=\\
  =  -\limsup_{m\to +\infty} \frac{-e_-(G',F^{-m}(G))}m=
  -(F^{-1})\Udim T.
\end{multline*}
\end{proof}

Finally, for the category $\Perf(A)$ where $A$ is a compact dg algebra we demonstrate that the $F$-dimension is finite provided that $F$ is not nilpotent.
\begin{predl}
\label{prop_Fdimfinite}
Suppose  $A$ is a compact dg algebra and $T=\Perf(A)$. Suppose the functor $F\colon T\to T$ is given by $F(-)=(-)\stackrel{\bbL}{\otimes }_A K$ for some  $K\in\Perf(A^{\op}\otimes A)$, and $F$ is not nilpotent. Then the upper and the lower $F$-dimensions of $T$ are finite.
\end{predl}
\begin{proof}
We can take $G=G'=A$, then  
$$F\Udim T=\limsup_{N\to\infty}\frac{-e_-(A,F^N(A))}N,\qquad F\Ldim T=\liminf_{N\to\infty}\frac{-e_+(A,F^N(A))}N.$$
Since $F$ is not nilpotent, the numbers $e_-(A,F^N(A))$ and $e_+(A,F^N(A))$ are finite for any~$N$.
Note that $e_-(A,F^N(A))=\inf K^{\otimes_A^{\bbL}N}$ and
$e_+(A,F^N(A))=\sup K^{\otimes_A^{\bbL}N}$. 
It remains to apply  Lemma \ref{lemma_lineargrowth} below.
\end{proof}

\begin{lemma} 
\label{lemma_lineargrowth} 
Let $A$ be a compact dg algebra and let $K$ be a perfect dg $A^{\op}\otimes A$-module.
Consider $\vert \sup K^{\otimes ^{\bbL}_AN}\vert $ and $\vert \inf K^{\otimes ^{\bbL}_AN}\vert $ as functions of $N$. Then they are bounded by a linear function.
\end{lemma}

\begin{proof} For some $A^{\op}\otimes A$-bimodule $K'$ we have $K\oplus K'\in [A\otimes A]$. It suffices to prove the statement for $K\oplus K'$. So we will assume that $K \in [A\otimes A]$ and thus $K$ has a finite filtration by dg $(A^{\op}\otimes A)$-submodules with associated graded isomorphic to a finite direct sum
$gr K\simeq \oplus (A\otimes A[d_i])$.
Then $K^{\otimes ^{\bbL}_AN}$ has a similar filtration with the associated graded isomorphic to $(gr K)^{\otimes ^{\bbL}_AN}$. So we may replace $K$ by $gr K$:
$$K= \oplus_{i=1}^n (A\otimes A[d_i]).$$
We have
$$K^{\otimes ^{\bbL}_AN}=\bigoplus_{i_1,\ldots,i_N=1}^n A\otimes A\otimes \ldots \otimes A[d_{i_1}+\ldots + d_{i_N}]=\bigoplus_{i_1,\ldots,i_N=1}^n A^{\otimes(N+1)}[d_{i_1}+\ldots +d_{i_N}].$$
Hence
$$\sup K^{\otimes ^{\bbL}_AN}=(N+1)\sup A-N(\min d_i)\quad \text{and}\quad
\inf K^{\otimes ^{\bbL}_AN}=(N+1)\inf A-N(\max d_i).$$
\end{proof}

\subsection{$F$-dimension and entropy} For the rest of this section we fix a field $\k$ and follow the same assumptions as in Section \ref{section_diagonal}. In particular all our categories will be $\k$-linear.

We will use results of \cite{DHKK} on the entropy of an endofunctor. For our applications we only need to consider the case when the category $T$ is equivalent to $\Perf(A)$ for a smooth and compact ($\k$-) dg algebra $A$. Moreover we will only consider endofunctors $F\colon \Perf(A)\to \Perf(A)$ of the form
$$F(-)=(-)\stackrel{\bbL}{\otimes }_AK$$
for a dg $A$-bimodule $K$. We will call such endofunctors \emph{tensor endofunctors}.

In \cite{DHKK} the authors define the \emph{entropy} of a triangulated endofunctor $F$ on a triangulated category $T$ as
a certain function $h_t(F)\colon \R\to \R\cup\{- \infty\}$. 
For $T =\Perf(A)$ where $A$ is a smooth and compact dg  algebra and for a tensor endofunctor $F\colon \Perf(A)\to \Perf(A)$ their definition is equivalent to the following.

\begin{definition}[{\cite[Definition 2.4, Theorem 2.6]{DHKK}}]
The entropy $h_t(F)$ of an endofunctor $F\colon T\to T$ on a $\k$-linear Ext-finite triangulated category with a generator is defined as 
\begin{equation}
\label{eq_entropy}
h_t(F)=\lim_{N\to\infty}\frac1N\ln\sum_n\dim\Hom^n(G,F^N(G))e^{-nt},
\end{equation}
where $G$ is a (arbitrary) generator of $T$. For a nilpotent functor $F$ the formula \eqref{eq_entropy} should read as $h_t(F)=-\infty$ for all $t\in\R$.
\end{definition}

\begin{theorem}[See {\cite[Lemma 2.5, Theorem 2.6]{DHKK}}]
\label{theorem_DHKK}
Let $T =\Perf(A)$ where $A$ is a smooth and compact dg  algebra. Let $F\colon T\to T$ be a tensor endofunctor. Let $G\in T$ be a generator.
For any $t\in \R$, the limit \eqref{eq_entropy} exists in $[-\infty, +\infty)$ and does not depend on $G$.
Moreover, for any generators $G,G'\in T$ one has
$$h_t(F)=\lim_{N\to\infty}\frac1N\ln\sum_n\dim\Hom^n(G,F^N(G'))e^{-nt}.$$
\end{theorem}


Given a tensor autoequivalence $F\colon \Perf(A)\to \Perf(A)$ for a smooth and compact dg algebra $A$, one can recover the upper and the lower $F$-dimension of $\Perf(A)$ from the entropy function $h_t(F)$, see Proposition \ref{prop_entropy-to-dimension} below.

\begin{lemma}
\label{lemma_bound}
Let $A$ be a smooth and compact dg algebra  and let $F\colon \Perf(A) \to\Perf(A)$ be a tensor endofunctor. Then for any generators $G,G'\in T$ there exists $c>0$
such that for any $N\ge 0$
$$\sum_n \dim\Hom^n(G,F^N(G'))\le c^N.$$
\end{lemma}
\begin{proof}
Entropy $h_0(F)$ is the limit
$$\lim_{N\to\infty}\frac1N\ln\sum_n\dim\Hom^n(G,F^N(G')),$$
which is finite (or $-\infty$) by Theorem~\ref{theorem_DHKK}.
Hence the sequence $\frac1N\ln\sum_n\dim\Hom^n(G,F^N(G'))$ is bounded above, say by a number $d$. Take $c=e^d$.
\end{proof}

\begin{predl}
\label{prop_entropy-to-dimension}
Let $A$ be a smooth and compact dg algebra  and let $F\colon \Perf(A)\to \Perf(A)$ be a tensor endofunctor. Then for any generators $G,G'\in \Perf(A)$ one has
\begin{gather*}
F \Udim \Perf(A)=\lim_{N\to\infty} \frac{-e_-(G,F^N(G'))}N=\lim_{t\to +\infty}h_t(F)/t,\\
F\Ldim \Perf(A)=\lim_{N\to\infty} \frac{-e_+(G,F^N(G'))}N=\lim_{t\to -\infty}h_t(F)/t.
\end{gather*}
Moreover, if $F$ is not nilpotent, then these limits are finite.
\end{predl}

\begin{proof} Let us prove the equalities in the first row. The equalities in the second row can be proved similarly.

One has for $t>0$
\begin{multline*}
\frac{-e_-(G,F^N(G'))}N\cdot t\le
\frac 1N\ln \sum_n\dim \Hom^n(G,F^N(G'))e^{-nt}\le \\
\le \frac 1N\ln \sum_n\dim \Hom^n(G,F^N(G'))e^{-e_-(G,F^N(G'))\cdot t}=\\
=\frac 1N(-e_-(G,F^N(G')\cdot t)+\frac1N\ln \sum_n\dim\Hom^n(G,F^N(G')))\le \\
\le \frac{-e_-(G,F^N(G'))}N \cdot t+\frac1N \ln c^N=
\frac{-e_-(G,F^N(G'))}N\cdot t+\ln c.
\end{multline*}
The last inequality here is by Lemma~\ref{lemma_bound}.
It follows that for any $t>0$
$$\limsup_{N\to+\infty} \frac{-e_-(G,F^N(G'))}N \le \frac 1t\limsup_{N\to+\infty}\frac 1N\ln \sum_n\dim \Hom^n(G,F^N(G'))e^{-nt}=h_t(F)/t$$
and
$$\liminf_{N\to+\infty} \frac{-e_-(G,F^N(G'))}N \ge -\frac{\ln c}t + \frac 1t\liminf_{N\to+\infty}\frac 1N\ln \sum_n\dim \Hom^n(G,F^N(G'))e^{-nt}=-\frac{\ln c}t + h_t(F)/t.$$
Since $t>0$ is arbitrary it follows that the limits $\lim_{N\to+\infty} \frac{-e_-(G,F^N(G'))}N$ and  $\lim_{t\to +\infty}h_t(F)/t$ exist and are equal. 
They are also equal to $F \Udim \Perf(A)$ by Definitions \ref{definition_Fdim1} and \ref{definition_Fdim2}.

The assertion about finiteness follows from Proposition~\ref{prop_Fdimfinite}.
\end{proof}






As a byproduct of our methods we also obtain the following result.

\begin{predl}
Let $A$ be a smooth and compact dg algebra. Let $F\colon \Perf(A)\to \Perf(A)$ be a tensor endofunctor. Assume that
 $F\Udim \Perf(A)=F\Ldim\Perf(A)$ (in particular, $F$ is not nilpotent). Then one has
$$h_t(F)=h_0(F)+(F\Udim \Perf(A))\cdot t.$$
\end{predl}
\begin{proof} Choose a generator $G$ of $\Perf(A)$.
For $t>0$ one has
\begin{multline*}
\frac 1N\ln \sum_n\dim \Hom^n(G,F^N(G))e^{-nt}\le
\frac 1N\ln \sum_n\dim \Hom^n(G,F^N(G))e^{-e_-(G,F^N(G))\cdot t}= \\
=\frac{-e_-(G,F^N(G))}N\cdot t+\frac 1N\ln \sum_n\dim \Hom^n(G,F^N(G)).
\end{multline*}
Passing to limits as $N\to\infty$, we get
$$h_t(F)\le (F\Udim \Perf(A))\cdot t + h_0(F).$$
Similarly we get that 
$$h_t(F)\ge (F\Ldim\Perf(A))\cdot t + h_0(F).$$
As $F\Ldim \Perf(A)=F\Udim \Perf(A)$, the statement follows. For $t<0$ the proof is analogous.
\end{proof}

The following theorem is a generalization of Theorem \ref{theorem_serre-in-families} for the upper and lower dimension of a category with respect to an endofunctor. The proof is exactly the same and we omit it.

\begin{theorem} 
\label{theorem_generaliz-serre-in-families}
Let $R$ be a commutative Noetherian ring and let $A$ be an $R$-h-projective smooth and compact dg $R$-algebra. Let $K\in \Perf(A^{\op}\otimes _R A)$ and consider the corresponding functor
$$\Phi _K\colon \Perf(A)\to \Perf(A).$$
For a point $x\in \Spec R$ denote by $A _x$ the smooth and compact dg $k(x)$-algebra $A \otimes _Rk(x)$ and consider the object $K_x:=K\stackrel{\bbL}{\otimes }_R k(x)\in D(A_x^{\op}\otimes A_x)$ with the induced functor
$$\Phi _{K_x}\colon \Perf(A_x)\to \Perf(A_x).$$
\begin{enumerate}
\item[(a)]
Then for any $c\in \R$ the subsets
\begin{align}
\label{eq_closedsubsets1} 
\{x\in \Spec R  & \mid \Phi_{K_x}  \Udim  \Perf(A _x)\ge c\}\quad \text{and} \\
\label{eq_closedsubsets2} 
\{x\in \Spec R  & \mid \Phi_{K_x}  \Ldim \Perf(A _x)\le c\}
\end{align}
are closed under specialization in $\Spec R$.

\item[(b)] Assume in addition that $A$ is concentrated in degree zero, i.e. $A$ is an $R$-algebra, and assume that $A$ is a projective finitely generated $R$-module.  Then the subset \eqref{eq_closedsubsets2} is closed in Zariski topology. That is  $\Phi_{K_x}  \Ldim \Perf(A_x)$ is a lower semi-continuous function on $\Spec R$. If moreover  $\Phi _K$ is an autoequivalence of $\Perf(A)$, then also the subset \eqref{eq_closedsubsets1} is closed, i.e.  $\Phi_{K_x}  \Udim\Perf(A_x)$ is an upper semi-continuous function on $\Spec R$.
\end{enumerate}
\end{theorem}

\section{Comparison of the three dimensions}
\label{section_compare}

It would be interesting to compare the different notions of dimension.
The Rouquier and diagonal dimensions are of similar nature, which allows one to compare them easily. In fact we know that $\Rdim  \Perf(A)\le \Ddim \Perf(A)$ (Proposition \ref{ddim-geq-rdim}). But we do not have a systematic procedure of comparing the Serre dimension with the other two. In \cite{El} the Serre and the Rouquier dimension are computed for many interesting examples. Here we present some of the results.

\begin{example}[See {\cite[Propositions 4.3, 4.4]{El}}]
Let $\Gamma$ be a nontrivial connected quiver without oriented cycles and $\k\Gamma$ be its path algebra over field $\k$. Then we have the following
\begin{itemize}
\item $\Rdim \Perf(\k\Gamma)=0$ if $\Gamma$ is a Dynkin quiver of types $A,D,E$;  $\Rdim \Perf(\k\Gamma)=1$ otherwise.
\item  $\Ddim \Perf(\k\Gamma)=1$.
\item  $0<\LSdim \Perf(\k\Gamma)=\USdim\Perf(\k\Gamma)<1$ if $\Gamma$ is a Dynkin quiver of types $A,D,E$; 
$\LSdim \Perf(\k\Gamma)=\USdim \Perf(\k\Gamma)=1$ otherwise. 
\end{itemize}
\end{example}

\begin{example}[See {\cite[Proposition 7.2]{El}}]
\label{example_maintable}
Let $B_2$ be the path algebra of the quiver $\bul\to\bul$ of type $A_2$ over some field $\k$, and 
$B_3$ be the path algebra of the quiver $\bul\to\bul\to \bul$ of type $A_3$. Denote 
$B_2^n=B_2\otimes B_2\otimes \ldots\otimes B_2=B_2^{\otimes n}$ and $B_3^n=B_3^{\otimes n}$.
Then one has the following 
\begin{gather*}
\begin{array}{|c||c|c|c|c|c|c|c|c|}
\hline
\text{Algebra} & B_2 & B_2^2 & B_2^3 & B_2^4 & B_2^5 & B_2^6& B_2^{2k} & B_2^{3k} \\ \hline
\LSdim,\USdim & 1/3 & 2/3 & 1& 4/3 & 5/3 & 2 & 2k/3 & k \\ \hline
\Rdim & 0 & 0 & 1& 1 & & 2 & \le k-1 & \ge k \\ 
\hline
\Ddim & 1 & 1 &  &  & &  & \le k & \ge k \\ 
\hline
\end{array}\\
\begin{array}{|c||c|c|c|c|c|c|c|}
\hline
\text{Algebra} & B_3 & B_3^2 & B_3^3 &  B_3^4 & B_3^{3k} & B_3^{3k+1} & B_3^{2k}\\ \hline
\LSdim,\USdim & 1/2 & 1& 3/2& 2 & \frac{3k}2 &\frac{3k+1}2 & k\\ \hline
\Rdim & 0& 1& 1& 2& \le 2k-1 & \le 2k & \ge k\\ 
\hline
\Ddim & 1& & & & \le 2k & \le 2k+1 & \ge k\\ 
\hline
\end{array} 
\end{gather*}
\end{example}

We remark that in all the above examples one has 
$$\Rdim\le \LSdim=\USdim\le \Ddim.$$

We make the following.
\begin{conjecture} \label{main-conj} 
Let $A$ be a smooth and compact dg algebra over a field. Suppose 
$\LSdim \Perf(A)=\USdim \Perf(A)$.
Then
\begin{equation}
\label{eq_4dimbis}
\Rdim  \Perf(A)\le \LSdim  \Perf(A)=\USdim  \Perf(A)\le \Ddim  \Perf(A).
\end{equation}
\end{conjecture}

In general, if $\LSdim\Perf(A)<\USdim\Perf(A)$, inequalities \eqref{eq_4dimbis} can fail. For instance, in Example~\ref{example_xyz} one has $1=\Rdim>\LSdim=0$ and 
$3=\USdim>2\ge \Ddim$.
In Example~\ref{example_orlov}
we have $1=\Rdim >\LSdim=1-w$ and $1+w=\USdim>1= \Ddim$ as soon as $w>0$.

Nevertheless, we do not know of an example where $\Rdim>\USdim$ or $\Ddim<\LSdim$.


\begin{thebibliography}{9999999}

\bibitem[Am07]{Am} C.\,Amiot, ``On the structure of triangulated categories with finitely many indecomposables'', 
Bull. Soc. math. France
135 (3), 2007, 435--474.

\bibitem[BF12]{BaFa} M.\,Ballard, D.\,Favero, "Hochschild dimensions of tilting objects" Int. Math. Res. Not. IMRN 2012, no. 11, 2607--2645.

\bibitem[BBD81]{BBD} A.\,Beilinson, J.\,Bernstein, P.\,Deligne, "Faisceaux pervers" (French) [Perverse sheaves] Analysis and topology on singular spaces, I (Luminy, 1981), 5--171, Ast\'{e}risque, 100, Soc. Math. France, Paris, 1982.

\bibitem[BK89]{BK1} A.\,Bondal, M.\,Kapranov, "Representable functors, Serre functors, and reconstructions" (Russian) Izv. Akad. Nauk SSSR Ser. Mat. 53 (1989), no. 6, 1183--1205, 1337; translation in Math. USSR-Izv. 35 (1990), no. 3, 519--541

\bibitem[BK90]{BK2} A.\,Bondal, M.\,Kapranov, "Framed triangulated categories" (Russian) Mat. Sb. 181 (1990), no. 5, 669--683; translation in Math. USSR-Sb. 70 (1991), no. 1, 93--107

\bibitem[BLL04]{BLL} A.\,Bondal, M.\,Larsen, V.\,Lunts,  "Grothendieck ring of pretriangulated categories" Int. Math. Res. Not. 2004, no. 29, 1461--1495.

\bibitem[BVdB03]{BVdB} A.\,Bondal, M.\,Van den Bergh, "Generators and representability of functors in commutative and noncommutative geometry" Mosc. Math. J. 3 (2003), no. 1, 1--36, 258.

\bibitem[DHKK13]{DHKK} G.\,Dimitrov, F.\,Haiden, L.\,Katzarkov, M.\,Kontsevich, ``Dynamical systems and categories'', arXiv.math: 1307.8418

\bibitem[Dr04]{Dr} V.\,Drinfeld, "DG quotients of DG categories" J. Algebra 272 (2004), no. 2, 643--691.

\bibitem[ELO09]{ELO} A.\,Efimov, V.\,Lunts, D.\, Orlov, "Deformation theory of objects in homotopy and derived categories I. General theory." Adv. Math. 222 (2009), no. 2, 359--401.

\bibitem[El20]{El} A.\,Elagin, ``Calculating dimension of triangulated categories: path algebras, their tensor powers and orbifold projective lines'', arXiv:2004.04694. 

\bibitem[Ga73]{Ga} P.\,Gabriel, "Indecomposable representations. II" Symposia Mathematica, Vol. XI (Convegno di Algebra Commutativa, INDAM, Rome, 1971), pp. 81--104. Academic Press, London, 1973

\bibitem[Ha18]{Han} N.\,Hanihara, ``Auslander correspondence for triangulated categories'', arXiv:1805.07585v2.

\bibitem[Ha77]{Ha} R.\,Hartshorne, "Algebraic Geometry", Springer, 1977.

\bibitem[Ke94]{Ke} B.\,Keller, "Deriving DG categories"
Ann. Sci. Ecole Norm. Sup. (4) 27 (1994), no. 1, 63--102.

\bibitem[Li05]{Lieb} M.\,Lieblich, "Moduli of perfect complexes on a proper morphism", Journal of Algebraic Geometry, (15)1, 2005.

\bibitem[LS14]{LS1} V.\,Lunts, O.\,Schn\"{u}rer, "Smoothness of equivariant derived categories" Proc. Lond. Math. Soc. (3) 108 (2014), no. 5, 1226--1276.

\bibitem[LS16]{LS2} V.\,Lunts, O.\,Schn\"{u}rer, "New enhancements of derived categories of coherent sheaves and applications" J. Algebra 446 (2016), 203--274.

\bibitem[LO10]{LO} V.\,Lunts, D.\,Orlov, "Uniqueness of enhancement for triangulated categories'', J.
Amer. Math. Soc. 23 (2010), no. 3, 853--908.


\bibitem[Or09]{Or} D.\,Orlov, "Remarks on generators and dimensions of triangulated categories" Mosc. Math. J. 9 (2009), no. 1, 153--159.

\bibitem[Ri84]{Ri} C.\,M.\,Ringel, ``Tame algebras and integral quadratic forms'', Springer, Berlin--Heidelberg--New York (1984). Lecture Notes in Mathematics 1099.

\bibitem[Ro08]{Rou}
R.\,Rouquier, ``Dimensions of triangulated categories'', {\it Journal of K-theory}, 1(2):193--256, 2008.

\bibitem[Sh]{Sh} D.\,Shklyarov, "On Serre duality for compact homologically smooth dg algebras", unpublished manuscript.

\bibitem[TT90]{ThTr} R.\,Thomason, T.\,Trobaugh, "Higher algebraic K-theory of schemes and of derived categories", The Grothendieck Festschrift, Vol. III, 247--435, Progr. Math., 88, Birkh\"{a}user Boston, Boston, MA, 1990.

\bibitem[To07]{To} B.\,To\"{e}n, "The homotopy theory of dg-categories and derived Morita theory" Invent. Math. 167 (2007), no. 3, 615--667.



\end{thebibliography}
\end{document}